\documentclass{amsart}
\usepackage[utf8]{inputenc}
\usepackage{amsmath, amsthm, amssymb, amsfonts}
\usepackage{bm}
\usepackage{dsfont}
\usepackage[utf8]{inputenc}
\usepackage{rotate}
\usepackage{amsthm}

\usepackage{tikz}
\usepackage{tikz-cd}
\usepackage[arrow,matrix,curve]{xy} 	
\usepackage{xcolor}
\usepackage{todonotes}
\usepackage{alltt} 
\usepackage{amsmath,amssymb}
\usepackage{url}
\usepackage{todonotes}
\usepackage[margin=1in]{geometry}

\usepackage{cleveref}

\usepackage{comment}

\usetikzlibrary{arrows.meta,calc,positioning}

\theoremstyle{plain}
\newtheorem{theorem}{Theorem}[section]
\newtheorem{proposition}[theorem]{Proposition}
\newtheorem{lemma}[theorem]{Lemma}
\newtheorem{corollary}[theorem]{Corollary}

\theoremstyle{definition}
\newtheorem{definition}[theorem]{Definition}

\newtheorem{remark}[theorem]{Remark}
\newtheorem{example}[theorem]{Example}

\newtheorem{specialtheorem}{Theorem}

\newcommand{\er}{\mathfrak{r}}
\newcommand{\uker}{\mathrm{uker}}

\newcommand{\I}{\mathrm{Id}}

\newcommand{\Id}{{\rm{Id}}}

\newcommand{\M}{{\rm{M}}}

\newcommand{\RR}{{\mathbb{R}}}
\newcommand{\ZZ}{{\mathbb{Z}}}
\newcommand{\NN}{{\mathbb{N}}}
\newcommand{\QQ}{{\mathbb{Q}}}

\DeclareMathOperator{\Tail}{Tail}
\DeclareMathOperator{\Per}{Per}

\title{Dimensions of Orbit Closures and Discrepancy for Dynamical $p$-adic Sequences}
\author{Keivan Mallahi Karai \and Christian Wei\ss{}}
\date{\today}

\begin{document}

\begin{abstract}
Classical discrepancy quantifies the irregularity of the distribution of a sequence in the unit interval. In this paper, we study the analogous notion for sequences in the ring of $p$-adic integers with a focus on the dynamically generated sequences.  We prove that the orbits of ergodic $1$-Lipschitz self-maps of $\ZZ_p^d$ attain the optimal order of discrepancy and hence form low-discrepancy sequences. We also obtain bounds on the growth of the size of orbits of polynomial self-maps of $f: \ZZ_p^d \to \ZZ_p^d$ modulo $p^n$ for $d>1$. As a consequence, we show that orbit closures of $f$ have box dimension either zero or one. Our approach relies on the introduction of strong fixed points for such maps, together with several decomposition results for matrices over $\ZZ_p$.

\end{abstract}
\maketitle

\section{Introduction and statement of results}
In the classical setting, a sequence in the unit interval is uniformly distributed if the proportion of its terms lying in any subinterval converges to the length of that subinterval. Independent samples from the uniform distribution have this property almost surely, but many deterministic sequences are uniformly distributed as well. Discrepancy theory refines this qualitative notion by quantifying the error in these approximations and thus measuring how closely a deterministic sequence mimics uniform random sampling. Schmidt’s celebrated lower bound \cite{Sch72} determines the optimal asymptotic order of discrepancy, motivating the study of low-discrepancy sequences, whose discrepancy attains this order. 

The study of uniformly distributed sequences has a long history, beginning with the work of Kronecker and Weyl on sequences of the form
$$
\bigl({f(n)}\bigr)_{n\in\NN},
$$
where $f\in\RR[x]$ has at least one nonconstant irrational coefficient and $\{\cdot \}$ denotes the fractional part. In the linear case, these are known as Kronecker sequences. Other classical constructions—including the van der Corput sequence, Halton sequences, and more general digital sequences—yield explicit deterministic sequences that are not only uniformly distributed but also have low discrepancy; see, for example, \cite{KN74}. Real low-discrepancy sequences can also be constructed using methods from dynamical systems; see, for instance, \cite{GrabnerHellekalekLiardet2012}. 

Uniform distribution and discrepancy have natural analogues for sequences in the ring of $p$-adic integers. Recall that $\ZZ_p$ may be viewed as the space of infinite base-$p$ digit strings, equipped with its compact topology and normalized Haar measure, or equivalently as the completion of $\ZZ$ with respect to the $p$-adic metric. In this setting, discrepancy measures how evenly a sequence visits residue classes modulo $p^k$, for $k\geq 1$, or, equivalently, how uniformly its terms are distributed among finite digit prefixes.

In this paper we study such questions for sequences defined dynamically. Formally, let $X=\mathbb{Z}_p^d$ (or $[0,1]^d$), and let $f: X \to X$ be a continuous function. Given an initial point $x_0=x\in X$, the associated dynamical sequence to $(f, x_0)$ is defined recursively by
\[ x_{n+1}=f(x_n), \quad n \ge 0.\]
We are interested in studying situations in which sequences $(x_n)_{n\ge 1}$ defined as above have low-discrepancy.  

A map on $\ZZ_p$ produces at every level $k$ a finite dynamical system modulo $p^k$. The discrepancy of an orbit measures how evenly these finite quotients are covered, uniformly over the levels  $k$. In this sense, $p$-adic discrepancy may be viewed as a quantitative form of uniform distribution for profinite dynamical systems: it asks not only whether an orbit is uniformly distributed, but how efficiently it explores the increasingly fine residue towers of $\ZZ_p$.

In this article, we will show that, when $X= \ZZ_p^d$, ergodic $1$-Lipschitz maps attain the optimal ($p$-adic) discrepancy bound $O(1/N)$. We establish a direct connection in $\ZZ_p^d$ between the qualitative property of ergodicity and the quantitative property that orbits form low-discrepancy sequences. This is indeed surprising as a qualitative property (ergodicity) implies a subtler quantitative property (low-discrepancy orbits). 

\begin{specialtheorem} \label{thm:ergodic:1d}
Let $f: \mathbb{Z}_p^d \to \mathbb{Z}_p^d$ be a measure-preserving $1$-Lipschitz function. Then the following are equivalent:
\begin{enumerate}
\item $f$ is ergodic. 
\item $f$ is uniquely ergodic. 
\item Every orbit is uniformly distributed. 
\item For every $x_0 \in \mathbb{Z}_p^d$ the sequence $(x_n)$ defined by the orbit
\[
x_n:= f^{(n)}(x_0) 
\]
is a low-discrepancy sequence.
\end{enumerate}
\end{specialtheorem}

We stress that this is in fact a feature of the $p$-adic integers which does not transfer to the Archimedean setting. To see this, we consider, for example, the Kronecker system
$f_\alpha: \mathbb{R^d}/\mathbb{Z^d} \to \mathbb{R^d}/\mathbb{Z^d}$ defined by
\[
f_\alpha(x)=x+\alpha.
\]
The resulting Kronecker sequence $x_n= \{ n \alpha \}$ is intimately connected with the continued fraction of $\alpha$. If and only if $\alpha$ has bounded partial quotients, then one obtains the optimal (real) discrepancy behavior $O(\log N/N)$. For the remaining Lebesgue almost all $\alpha$, it is known that the slightly weaker bound $O(\log N \log \log N/N)$ holds. We refer the reader to \cite[Section 1.4]{Tichy} for a concise discussion of the discrepancy of Kronecker sequences for general $d \ge 1$. 

Theorem~\ref{thm:ergodic:1d} implies that when $f: \ZZ_p^d \to \ZZ_p^d$ is ergodic, one automatically obtains the quantitative fact that its orbits form low-discrepancy sequences. Conversely, in the absence of ergodicity, one may seek to quantify the failure of uniform distribution. Our next theorem shows that when $d>1$, polynomial maps are not only non-ergodic, but their orbits are extremely far from being uniformly distributed. Recall that a polynomial map $f: \ZZ_p^d \to \ZZ_p^d$ is a function of the form
\[
f(x) = \bigl(f_1(x), \dots, f_d(x)\bigr),
\]
where each $f_j(x)$ is a polynomial in the coordinates $x_1, \dots, x_d$ of $x$, with coefficients in $\ZZ_p$.

This observation, formulated in Theorem~\ref{thm:Hausdorff}, is in stark contrast to the real case, where even multi-dimensional Kronecker sequences $\{ n\alpha_1, n\alpha_2,\ldots, n\alpha_d)\}$ are uniformly distributed, whenever $1,\alpha_1,\ldots,\alpha_d$ are linearly independent, see again \cite{Tichy}.

\begin{specialtheorem}[Box dimension alternative] \label{thm:Hausdorff}
Let $f: \mathbb{Z}_p^d \to \mathbb{Z}_p^d$ be a polynomial map, and let
$x_0 \in \mathbb{Z}_p^d$ be arbitrary. Let
\[
E=\overline{\{f^{(j)}(x_0): j \geq 0\}}
\]
be the closure of the orbit of $x_0$. Then
\[
\dim_B(E) \in \{0,1\}.
\]
\end{specialtheorem}

Here and throughout, $\dim_B$ denotes the box dimension; see Section~\ref{boxdim} for its definition. Theorem~\ref{thm:Hausdorff} shows that, in any ambient dimension, polynomial dynamical sequences are asymptotically at most one-dimensional. In fact, Theorem~\ref{thm:Hausdorff} follows from a more precise, non-asymptotic result, which we now state.

For each $n\geq 1$, let
$$
f_n:(\ZZ/p^n\ZZ)^d\longrightarrow(\ZZ/p^n\ZZ)^d
$$
denote the map induced by $f$. We write $\Per_{f_n}(x)$ for the period of the eventual periodic part of the $f_n$-orbit of $x$, and $\Tail_{f_n}(x)$ for the length of its preperiodic part; see Section~\ref{sec:prelim} for the precise definitions. We have the following theorem.

\begin{specialtheorem}[Effective one-dimensionality of orbits of polynomial maps]\label{them:finitary}
Let $f: \mathbb{Z}_p^d \to \mathbb{Z}_p^d$ be a polynomial map with coefficients in $\ZZ_p$, and let
$x \in \mathbb{Z}_p^d$ be arbitrary. Then we have 

\begin{enumerate}
\item Either there exists a constant $C_0$ such that $ \Per_{f_n}(x) \le C_0$ for all $n \ge 1$, or there exist  $C_1,C_2>1$ such that for all $n \ge 1$ we have 
\[ C_1 p^n \le \Per_{f_n}(x) \le C_2 p^n. \]

\item There exists $C_3 \ge 0$ such that for all $n \ge 1$, we have 
\[ \Tail_{f_n}(x) \le nC_3. \]
\end{enumerate}
In particular, the size of the orbit of $x$ under $f_n$ is at most $O(p^n)$. 

\end{specialtheorem}

Part (1) of this theorem (growth of periods) will be proven in Theorem~\ref{lem:period_fixed_points}. A more precise version of part (2) is proven in 
Theorem~\ref{thm:preperiod-growth}.

Theorem~\ref{thm:Hausdorff} concerns a particular class of self-maps of $\ZZ_p^d$, namely, the $1$-Lipschitz maps. Such maps have been studied extensively in the literature; see, for example, \cite{Ana94, Ana06, Jeo13, Jeo22}. Some restriction on the class of maps under consideration is necessary for any result of this kind that distinguishes between the cases $d=1$ and $d>1$. Indeed, there exists a homeomorphism
$\phi:\ZZ_p\longrightarrow \ZZ_p^d $ such that the push-forward under $\phi$ of the Haar measure on $\ZZ_p$ is the Haar measure on $\ZZ_p^d$. Consequently, every continuous self-map $f$ of $\ZZ_p$ is topologically conjugate to the continuous self-map $\phi\circ f\circ\phi^{-1}$ of $\ZZ_p^d$.

\medskip
The paper is organized as follows. In Section~\ref{sec:prelim}, we introduce the algebraic notions used throughout the paper, including the ring of $p$-adic integers $\ZZ_p$, its geometric structure, and $1$-Lipschitz maps. In Section~\ref{sec:ud_and_discrepancy}, we discuss uniform distribution and discrepancy in $\ZZ_p^d$ and prove Theorem~\ref{thm:ergodic:1d}.
Section~\ref{sec:linalg} establishes several results concerning linear algebra over $\ZZ_p$. Among them is a matrix decomposition result that may be of independent interest. These results are then applied in Section~\ref{sec:dynamics} to analyze the dynamics of multidimensional polynomial maps on $\ZZ_p^d$. One key notion introduced and studied in this section is the notion of strong fixed point.  This section also contains the proofs of Theorems~\ref{thm:Hausdorff} and~\ref{them:finitary}.
Finally, in Section~\ref{sec:examples}, we construct ergodic maps on $\ZZ_p^d$ for every $d\geq 1$, thereby obtaining explicit examples of multidimensional $p$-adic low-discrepancy sequences.

\medskip
{\it Acknowledgments}. The authors would like to thank Anke Pohl for drawing our attention to our shared research interests, which initiated this collaboration.

\section{Preliminaries from Algebra and Dynamical Systems}\label{sec:prelim}
In this section, we introduce some notation and define several basic notions on the $p$-adic numbers and their geometry as well as on dynamical systems that will be used throughout the sequel. Readers interested in further details may consult \cite{Neu99}.

\subsection{The ring of $p$-adic integers and its metric} We will start with a quick review of the construction and some basic properties of the ring of $p$-adic integers. Let $p$ be a prime number.  The ring $\ZZ_p$ of $p$-adic integers may be viewed both algebraically and
metrically. Algebraically, it is constructed as the inverse limit
\[ \ZZ_p=\varprojlim_{n} \ZZ/p^n\ZZ \]
with respect to the natural reduction maps $\theta_{m,n}: \ZZ/p^{m}\ZZ \to \ZZ/p^n\ZZ$ for $ m >n$. Thus, an element of $\ZZ_p$ can be regarded as  a compatible
system of residue classes
\[ (x_n)_{n \in \mathbb{N}}, \qquad x_n \in \ZZ/p^n\ZZ, \qquad x_{m} \equiv x_n \pmod{p^{\min(m,n)} }. \]
Equivalently, every $p$-adic integer has a unique $p$-adic representation as 
\begin{equation}\label{$p$-adic}
 x=a_0+a_1p+a_2p^2+\cdots, \qquad a_i \in \{0,1,\ldots,p-1\}. 
\end{equation}
In this case, we will define the \textit{reduction mod}  $p^n$ map $\pi_n: \ZZ_p \to \ZZ/p^n \ZZ$ by 
\[
\pi_n(x)=x_n:=a_0+a_1p+ \cdots + a_{n-1}p^{n-1}.
\]
When $ \theta_{m,n}(y)=x$, we call $y$ a lift of $x$. If $ y$ is represented by $(a_0+a_1p+ \cdots + a_{n-1}p^{m-1})$, then $x$ is represented by $a_0+a_1p+ \cdots + a_{n-1}p^{n-1}$ and we can write
\[ y = x+ p^n z, \qquad z= a_{n}+ \cdots + a_{m}p^{m-n-1}. \]
Often we abuse notation and use the shorthand $ y= x+ p^n z$ with $z \in \ZZ/p^{m-n} \ZZ$. 

\medskip 

\noindent
For a non-zero $x \in \ZZ_p$,  we write $v_p(x)$ for the $p$-adic valuation of $x$ which is defined as the least $n$ such that $a_n \neq 0$ in \eqref{$p$-adic}. The metric on $\ZZ_p$ is induced by the $p$-adic absolute value. Namely, for
$x,y \in \ZZ_p$ one sets
\[
d_p(x,y)=|x-y|_p=p^{-v_p(x-y)}.
\]
Similarly, $\mathbb Z_p^d$ for $d \geq 1$ is equipped with the natural $p$-adic metric
\[
d_p(x,y)=\max_{1\leq i\leq d}|x_i-y_i|_p.
\]
By a slight abuse of notation, we will write $d_p$ without specifying the dimension. In each application, the relevant dimension $d$ will either be clear from the context or the statement under consideration will hold for all $d$.

The following properties of $d_p$ will be used in the sequel. First, $d_p$ is an ultrametric, i.e., it satisfies $d_p(x,z) \le \max ( d_p(x, y), d_p(y, z))$ for all $x, y, z \in \ZZ_p$. Second, for each $k \geq 1$, the ball of radius $p^{-k}$ centered at $z \in \ZZ_p$ is defined by
\[
\textrm{Disc}_p(z,p^{-k}) := \left\{ x \in \mathbb{Z}_p \, : \, \left| x - z \right|_p \leq p^{-k} \right\} =
z+p^n \ZZ_p. 
\]

Note that $\textrm{Disc}_p(z,p^{-k})$ is a coset of the subgroup $p^k \ZZ_p$ in $\ZZ_p$.  More generally,  for $z=(z_1,\ldots,z_d)\in \mathbb{Z}_p^d$ and $k=(k_1,\ldots,k_d)\in \mathbb{Z}_{\ge 0}^d$, we define the $p$-adic polydisc centered at $z$ with resolution vector $k$ by
 
\[
\textrm{Disc}_p(z,p^{-k}) := \left\{ x \in \mathbb{Z}_p^d \, : \, \left| x_i - z_i \right|_p \leq p^{-k_i} \; \mathrm{for \ all} \; i=1,\ldots,d \right\} 
\]

It is easy to see that, for every $d \geq 1$, the space $\mathbb Z_p^d$ is compact, totally disconnected, and has no isolated points. Hence $\mathbb Z_p^d$ is homeomorphic to the Cantor set, which may be identified with $\{0,1\}^{\mathbb N}$.

\medskip

Suppose $x \in \ZZ_p^d$ and let $n \ge 1$.  One can see that  $\textrm{Disc}_p(z,p^{-n})$ decomposes into a union of $p^{d}$ balls of radius 
$p^{-(n+1)}$. These balls are in bijection with the cosets of $p^{n+1}\ZZ_p^d$ in $p^n \ZZ_p^d$.  Let $X$ denote the set of these balls. Note that every element of $X$ is of the form $B=\textrm{Disc}_p(y,p^{-(n+1)})$ for some $y \equiv x \pmod{ p^n }$. Writing 
$y=x+p^nt$, one observes that $B$ depends only on the congruence class of $t \in \ZZ_p^d$ modulo $p$.  Hence, we can define a bijection 
\begin{equation}\label{tau}
 \tau: X \to (\ZZ/p\ZZ)^d, \quad \tau (y)= [t] \in p^n\ZZ_p^d/p^{n+1}\ZZ_p^d \simeq
(\ZZ/p\ZZ)^d. 
\end{equation}  
Throughout the remainder of this article, this bijection will be used to shift between these two viewpoints. 

\subsection{Lipschitz maps on $\ZZ_p^d$}
It follows from the definition of the $p$-adic metric that a function $f: \mathbb{Z}_p^d \to \mathbb{Z}_p^d$ is continuous at $x_0 \in \mathbb{Z}_p$ if and only if for every $m \in \mathbb{N}$, there exists an integer $R  \in \mathbb{N}$ such that for every $x \in \ZZ_p^d$ satisfying $x \equiv x_0 \mod p^R$ it holds that $f(x) \equiv f(x_0) \mod p^m$. If one can choose $R(m) = m$ for all $m \in \mathbb{N}$ and all values of $x_0$, then $f$ is called a $1$-Lipschitz function. Equivalently, $f$ is $1$-Lipschitz if for all $x, y \in \ZZ_p^d$ we have 
\[
\left| f(x)-f(y) \right|_p \leq \left| x-y\right|_p.
\]
The class of $1$-Lipschitz functions has been extensively studied in the literature, see e.g. \cite{Ana94,Jeo13,Jeo22}. Important examples of $1$-Lipschitz functions are polynomials $f: \mathbb{Z}_p \to \mathbb{Z}_p$. Note that when $f$ is $1$-Lipschitz, then the image 
of $\textrm{Disc}_p(z,p^{-k})$ under $f$ is included in a unique ball of radius $p^{-k}$. 

\medskip

\subsection{Measures and the box dimension}\label{boxdim} One can also equip $\ZZ_p^d$ with its normalized Haar probability measure. This measure is uniquely characterized by the property that each residue class $a+p^k \ZZ_p^d$ has measure $p^{-dk}$ for $k \ge 0$ and $a \in \ZZ_p^d$. We denote this measure by $\mu_d$. One can show that the measure of a polydisc is given by 
\[
\mu_d(\textrm{Disc}_p(z,p^{-k})) = \prod_{i=1}^d p^{-k_i}.
\]

Let $E \subseteq \ZZ_p^d$. 
For each $n \geq 0$, let $N_n(E)$ denote the number of residue classes modulo $p^n$ which intersect $E$. In other words, set 
$N_n(E)=\lvert\pi_n(E)\rvert.$ 
The {\it box, or Minkowski, dimension} of $E$ is, whenever the limit defined by
\[
\dim_B(E)=\lim_{n\to\infty}\frac{\log N_n(E)}{n \log p}= \lim_{n\to\infty}\frac{\log |
\pi_n(E)|}{n \log p}
\]
exists. More generally, replacing the limit by $\limsup$ and $\liminf$ gives the upper and lower box dimensions of $E$, respectively. It is immediate from the definition that for $ E \subseteq \ZZ_p^d$, we have $ 0 \le \dim_B(E) \le d$. 

\medskip

\subsection{Orbits for mappings of finite sets} Let $S$ be a set, and  $f: S \to S$ be an arbitrary function. The orbit of $s$ under $f$ defined by
\[
\mathcal{O}(s):= \{ s, f(s), f^2(s), \dots\ \}. 
\]
Here and in the sequel, $f^k$ stands for the $k$-th iterate of $f$. An element $s \in S$ is called periodic if there exists $T \ge 1$ such that $f^T(s)=s$.  The least value of $T$ is called the period of $s$. 
When $T=1$, we say that $s$ is a fixed point for $f$. Note that this point is the unique $f$-fixed point in $\mathcal{O}(s)$.

When $S$ is finite, not all elements $f^j(s), j \ge 0$ can be distinct. Let $D$ be the least non-negative integer with the property that $f^D(s)$ is periodic, and let $T$ denote its period. More precisely, set $T$ to be the least positive integer such that $f^D(s)= f^{D+T}(s)$. Hence, the orbit of $s$ can be written as a disjoint union 
\[ 
\mathcal{O}(s)= \{ f^j(s): \ 0 \le j \le D-1 \} \cup \{ f^j(s): \ D \le j \le D+T-1 \}. 
\]
The first $D$ elements are referred to as the {\it pre-periodic part of the orbit} and the remaining $T$ elements are the {\it periodic} part of the orbit. We denote the value of $D$ by $\Tail_f(s)$ and the value of $T$ by $\Per_f(s)$. 
It is clear from the description above that the period $T$ is an invariant of the orbit, and that every element $f^j(s)$ with $j \ge D$ is periodic and has the same period $T$. More generally, one can see that if $\mathcal{O}(s_1) \cap \mathcal{O}(s_2) \neq \emptyset$, then $\Per_{f}(s_1)= \Per_f(s_2)$.

\medskip

For the sake of completion, we provide the proof of the following elementary lemma.

\begin{lemma}\label{sizeofperiod}
Let $S$ be a finite set and let $f:S \to S$ be a map and $s \in S$ have period $T \ge 1$. For every point $t$ in the $f$-orbit of $s$, and every $m \ge 1$ we have
\[
\frac{T}{\min(m,T)} \leq \frac{T}{\gcd(m,T)} = \Per_{f^m}(t) \leq T.
\]
\end{lemma}

\begin{proof}
Let $T=\Per_f(s)$. By definition, the forward $f$-orbit of $s$ eventually enters a cycle of length $T$. Since $t$ lies on the same forward orbit, the forward $f$-orbit of $t$ also eventually enters this same cycle.

Let $C$ be this $f$-cycle. Thus, for every $x \in C$, we have $f^T(x)=x$ and $T$ is the smallest positive integer with this property. Now consider the map $f^m$. Since the $f$-orbit of $t$ eventually enters $C$, the $f^m$-orbit of $t$ also eventually enters $C$. It remains to compute the period of a point of $C$ under $f^m$. Fix $x \in C$. Its period under $f^m$ is the smallest positive integer $L$ such that $(f^m)^L(x)=x$, which is equivalent to $f^{mL}(x)=x$. Since $x$ has period $T$ under $f$, this happens if and only if  $T \mid mL.$ Therefore, $\Per_{f^m}(t)=\frac{T}{\gcd(T,m)}.$ The claim follows from $ 1 \leq \gcd(m, T) \le \min(m, T)$. 
\end{proof}

\section{Uniform Distribution, Discrepancy and Proof of Theorem~\ref{thm:ergodic:1d}} \label{sec:ud_and_discrepancy}
The goal of this section is to introduce the discrepancy of $p$-adic sequences, with a particular emphasis on dynamically generated sequences. Readers interested in more details are referred to \cite{KN74, Mei68, Som22, Wei25}. Moreover, we prove Theorem~\ref{thm:ergodic:1d} at the end of the section.

\subsection{The $p$-adic discrepancy}
The $p$-adic discrepancy of a sequence $(x_n)_{n \in \mathbb{N}}$ in the $p$-adic integers $\mathbb{Z}_p$ is the sequence $(D_N(x_n))_{N \in \NN}$ defined by
\[
D_N(x_n) := \sup_{z \in \mathbb{Z}_p, k \in \mathbb{N}_0} \left| \frac{\#\left( \textrm{Disc}_p(z,p^{-k}) \cap \{ x_1,\ldots,x_N \}\right)}{N} -  p^{-k} \right|, \qquad N \ge 1.
\]
This definition generalizes to the $p$-adic setting the more widely studied notion of discrepancy for sequences in the real unit interval, for the definition and basic properties see \cite{Tichy}. Recall that a sequence $(x_n)_{n \in \mathbb{N}}$ in $\mathbb{Z}_p$ is {\it uniformly distributed} if 
\[
\lim_{N \to \infty} D_N(x_n) = 0.
\]
As in the classical setting of sequences in the unit interval, discrepancy theory may be viewed as a quantitative refinement of the qualitative notion of uniform distribution.

A fundamental problem in real discrepancy theory is to establish optimal lower bounds for the discrepancy of sequences and to construct sequences for which these bounds are attained. For sequences in $[0,1]$, it is known that the optimal order of magnitude for the discrepancy $D_N(x_n)$ is $\frac{\log N}{N}$, see \cite{Sch72}. This is a highly non-trivial result, improving on the elementary lower bound $D_N(x_n)\gg 1/N$. In contrast with the real setting, it is well-known that there exist sequences in $\ZZ_p$ for which the trivial lower bound is sharp $D_N(x_n) \leq \frac{C}{N}$.

\begin{proposition}
Let $(x_n)_{n \in \NN}$ be a sequence in $\ZZ_p$. Then $D_N(x_n) \ge 1/N$ for all $N \in \mathbb{N}$. Moreover, there are sequences with 
\begin{equation}\label{low}
D_N(x_n) \le \frac{C}{N}
\end{equation}
for all $N \ge 1$ and some $C \geq 1$. 
\end{proposition}
Any sequence $(x_n)_{n \in \NN}$ in $\ZZ_p$ with $ D_N(x_n) \le C/N$ for some $C>0$ will be referred to as a  $p$-adic \textit{low-discrepancy} sequence. Note that the term $p^{-k}$ in the definition of the discrepancy equals the Haar measure $\mu(\textrm{Disc}_p(z,p^{-k}))$. For the simplest case of affine linear sequences $x_n=a+bn$ with $a,b \in \mathbb{Z}_p$ it is easy to show that it is uniformly distributed if and only if $b$ is a unit, see e.g. \cite{Ana94,Cug62}. This criterion has been generalized to arbitrary $1$-Lipschitz functions in \cite[Proposition~6.26]{Jeo22} which constitutes a good reason to study this class of maps. It is also natural to analyze which sequences defined via polynomial functions $f: \mathbb{Z}_p \to \mathbb{Z}_p$, which are a subclass of $1$-Lipschitz functions, satisfy \eqref{low} for $x_n = f^{(n)}(x_0)$. This question has been studied in \cite{Wei25}. 
 
\begin{theorem} \label{thm:permutation} Let $f(x) \in \ZZ_p[x]$ be a polynomial and $x_0 \in \ZZ_p$. Then $(x_n)_{n \in \NN} = (f^{(n)}(x_0))_{n \in \mathbb{N}}$ satisfies \eqref{low} if and only if $f$ is a permutation polynomial $\bmod\ p^2$, i.e. it induces a bijection on $\mathbb{Z}/p^2\mathbb{Z}$.
\end{theorem}
To the best of the authors' knowledge, Theorem~\ref{thm:permutation} covers all explicit examples of $p$-adic low-discrepancy sequences besides the rather abstract result in \cite{Bee69} which have been previously found. 

More generally, the concepts of uniform distribution and low-discrepancy also transfer to the multi-dimensional case $\mathbb{Z}_p^d$.  Given a sequence $(x_n)_{n \in \NN}$ in $\ZZ_p^d$, its $d$-dimensional $p$-adic discrepancy is defined by
\[
D_N(x_n) := \sup_{z \in \mathbb{Z}_p^d, k \in \mathbb{N}_0^d} \left| \frac{\#\left( \textrm{Disc}_p(z,p^{-k}) \cap \{ x_1,\ldots,x_N \}\right)}{N} - p^{-dk} \right|.
\]
At first glance, it may come as a surprise that also in the multi-dimensional setting there are sequences which satisfy $D_N(x_n) \leq \frac{C}{N}$ and that the notion of low-discrepancy thus remains valid also for $\ZZ_p^d$ without any dependence on the dimension.

\subsection{Dynamical systems of $p$-adic origin} 
We will now set up the basic notation and definitions for $p$-adic dynamical systems.  A \textit{$p$-adic dynamical system} is a triple $(\mathbb{Z}_p^d,\mu_p^d,f)$, where $f: \mathbb{Z}_p^d \to \mathbb{Z}_p^d$ is a measurable function with respect to the Borel $\sigma$-algebra and the Haar measure $\mu_p^d$. We say that $f$ is \textit{measure-preserving} if $\mu(f^{-1}(S)) = \mu(S)$ holds for each measurable subset $S \subset \mathbb{Z}_p^d$. A measure-preserving function is said to be \textit{ergodic} if it has no proper invariant subsets, i.e., if either $\mu(S)=1$ or $\mu(S)=0$ holds for any measurable subset $S \subset \mathbb{Z}_p$ such that $f^{-1}(S)=S$.
We say that $f$ is transitive $\bmod\ p^n$ if the above sequence forms a single cycle in $(\mathbb{Z}_p/p^n\mathbb{Z}_p)^d$. Much of what follows will be based on \cite[Proposition~4.1]{Ana06}.
\begin{theorem}[Anashin] \label{thm:Anashin}
A compatible function $f: \mathbb{Z}_p^d \to \mathbb{Z}_p^d$ is ergodic if and only if it is transitive $\bmod\ p^n$ for all $n \in \mathbb{N}$.    
\end{theorem}
 Here, compatible means that $f$ descends to a function on the quotient, $f_n: (\mathbb{Z}_p / p^n \mathbb{Z}_p)^d \to (\mathbb{Z}_p / p^n \mathbb{Z}_p)^d$, which is natural multi-dimensional analogue of being $1$-Lipschitz. 

\subsection{Proof of Theorem~\ref{thm:ergodic:1d}}
Having established the necessary preliminaries, we now proceed to the proof of Theorem~\ref{thm:ergodic:1d}.
\begin{proof}[Proof of Theorem~\ref{thm:ergodic:1d}]
It is clear that (4) implies (3) and (2) implies (1). That (3) implies (2) is a general fact which can be proven in the following way: In view of the ergodic decomposition of measures, it suffices to show that there exists a unique $f$-invariant measure $\mu$. In fact, if $\mu$ is an $f$-invariant measure, then for $\mu$-almost every point $x \in \ZZ_p^d$ and every continuous function $g: \ZZ_p^d \to \RR$ the sequence of Birkhoff averages $ \frac{1}{n} \sum_{i=1}^{n}g(x_n)$ converges to $ \int g \ d\mu$. On the other hand, (3) implies that the same Birkhoff sum must converge to $\int g d\nu$, where $\nu$ is the Haar measure on $\ZZ_p^d$. It follows that $\nu=\mu$, proving the claim. 

Thus, it remains to prove that (1) implies (4). Let $N \in \mathbb{N}$ and consider an arbitrary $p$-adic disc $\mathrm{Disc}_p(z,p^{-k})$. Within each block of $p^{dk}$ consecutive elements of the sequence, there is exactly one element contained in $\mathrm{Disc}_p(z,p^{-k})$. This follows from the one-to-one correspondence between such discs and the residue classes in $\mathbb{Z}_p^d/p^k\mathbb{Z}_p^d$ together with the ergodicity of $f$, which ensures that each residue class in $\mathbb{Z}_p^d /p^k\mathbb{Z}_p^d$ is represented exactly once within every block of length $p^{dk}$ according to Theorem~\ref{thm:Anashin}. If we set $r= \lfloor N/p^{dk} \rfloor$, there are thus either $r$ or $r+1$ elements contained in $\mathrm{Disc}_p(z,p^{-k})$. Hence, it follows that
\begin{align*}
\left| \frac{\#\left( \textrm{Disc}_p(z,p^{-k}) \cap \{ x_1,\ldots,x_N \}\right)}{N} - \frac{1}{p^{dk}} \right| & \leq \frac{1}{N}.
\end{align*}
It follows that the trajectory is a low-discrepancy sequence. 
\end{proof}

 Theorem~\ref{thm:ergodic:1d} reduces the question of constructing dynamical low-discrepancy sequences to that of finding ergodic transformations of $\ZZ_p^d$. 
A complete characterization of {\it linear} ergodic functions can be given as follows. For $a, b \in \ZZ_p$, the linear function $f_{a,b}(x)=ax+b$ is ergodic if and only if $b$ and $p$ are relatively prime and $a \equiv 1 \pmod{p}$ for $p>2$ and $a \equiv 1 \pmod{4}$, $b$ odd if $p=2$ respectively, see \cite[Proposition~1.5]{Ana06}. The ergodicity of $1$-dimensional affine maps is discussed in detail in \cite{FLYZ07} and for polynomials in \cite{FL11}.

Furthermore, note that $f(x)=ax+b$ is a permutation polynomial $\bmod\ p$ (but not necesserily $\bmod\ p^2$) if $a$ and $p$ are relatively prime, but it is only ergodic under the above mentioned conditions. Hence, for linear transformations, being ergodic is strictly stronger than being a permutation polynomial $\bmod\ p$, compare Theorem~\ref{thm:permutation}. At the same time, if $f$ is an ergodic polynomial, then it is necessarily also a permutation polynomial $\bmod\ p^2$: As $f$ is ergodic, it is also transitive $\bmod\ p^2$, see Theorem~\ref{thm:Anashin}. Hence $f^i(x)$ for $i=1,\ldots,p^2$ lie in all residue classes $\bmod\ p^2$. Since $f$ is $1$-Lipschitz, it satisfies $f(i) \equiv f(i + p^2) \bmod\ p^2$. Hence the values $f(1),\ldots,f(p^2)$ also occupy all residue classes $\bmod\ p^2$. In other words, $f$ is a permutation polynomial $\bmod\ p^2$. The stronger claim that $x_n = f(n)$ is a low-discrepancy sequence can be shown by proving that $f$ is then a permutation polynomial $\bmod\ p^k$ for all $k \in \mathbb{N}$, see \cite{Noe65,Wei25}.

\section{Idempotent decomposition over $\ZZ_p$} \label{sec:linalg}
The main goal of this section is to prove Theorem \ref{lem:decom}, which is a statement about decomposition of large power of matrices over $\ZZ_p$. Because $\ZZ_p$ is not a field and its field of fractions $\QQ_p$ is not algebraically closed, standard arguments from linear algebra over algebraically closed fields do not always apply and must be replaced by more careful considerations. The results obtained here are also of independent interest.

\begin{lemma}\label{lem:inv}
Let $C \in \M_d(\ZZ_p)$ be such that the reduction modulo $p$ $\pi_1(C)$ of $C$ has non-zero determinant. Then $C$ is invertible in 
$\M_d(\ZZ_p)$, that is, there exists $C' \in \M_d(\ZZ_p)$ with $CC'=C'C=I$. 
\end{lemma}
\begin{proof}
Since $\det$ is defined by an integral polynomial in the coefficients of a matrix and the reduction map $ \pi_1: \ZZ_p \to\ZZ/p\ZZ$ is a ring homomorphism, we have that $ \pi_1(\det C)= \det ( \pi_1(C))$. Hence, the assumption $\pi_1(\det C) \neq 0$ implies that $\det C \not\in p \ZZ_p$, and in particular $C$ has an inverse in 
$\M_d(\QQ_p)$ given by 
\[ C^{-1}= \frac{1}{\det C } \textrm{adj}(C) \]
Note that all entries of $\textrm{adj}(C)$ are polynomials with integer coefficients in entries of $C$, and hence $\textrm{adj}(C) \in \M_d(\ZZ_p)$. Since $\det C$ is not divisible by $p$, we deduce $(\det C)^{-1} \in \ZZ_p$. In conclusion,  $C^{-1}\in \M_d(\ZZ_p)$. 
\end{proof}

\noindent
{\it Rational functions of a matrix with entries in  $\ZZ_p^d$}. Let $B \in \M_d(\ZZ_p)$ and let $f(X)$ be a polynomial with coefficients in $\ZZ_p$. Then we say that $C:=f(B) \in \M_d(\ZZ_p)$ is a polynomial function of $B$. It is clear that if $C_1$ and $C_2$ are polynomial functions of $B$ then $C_1C_2=C_2C_1$. Suppose $f(X)$ and $g(X)$ are polynomial functions with coefficient in $\ZZ_p$ and $\det g(B) \not\equiv 0 \pmod{p}$. We say that $C:= f(B) g(B)^{-1}$ is a rational function of $B$.  Note that that $g(B)^{-1} \in \M_d(\ZZ_p)$ by Lemma~\ref{lem:inv}, and hence $C$ is an elements of $\M_d(\ZZ_p)$. If $C_1$ and $C_2$ are rational functions of $B$, then it is easy to see that $C_1$ and $C_2$ commute. Moreover, $C_1+C_2$ and  $C_1C_2$ are both rational functions of $B$. In particular, rational functions in $B$ form a 
$\ZZ_p$-algebra. 

\medskip

\noindent
\subsection{A decomoposition for matrices in $\M_d(\ZZ_p)$}
Recall that, by the Ellis–Numakura lemma, every compact semigroup contains an idempotent element. Given a matrix $A \in \M_d(\ZZ/p\ZZ)$, applying this to the semigroup defined as the closure of $ \{ A^n: \, n \ge 1\}$, we deduce that there exist a subsequence $n_i$ with $A^{n_i} \to S$ and $S^2=S$. In the following lemma, we will prove a finitary version of this fact, and apply it to introduce invariants of $A$ that will be used in the remainder of this paper.

\begin{lemma}\label{lem:decom}
Let $A\in \M_d(\mathbb Z_p)$. Then there exist an integer $m\ge 1$,
an idempotent matrix
\[
S\in \M_d(\mathbb Z_p), \qquad S^2=S,
\]
and a matrix $M \in \M_d(\mathbb Z_p)$ such that
\[
A^m=S+pM
\]
and
\[
SM=MS.
\]
In particular, the reduction $ \overline{S}$ of $S$ modulo $p$ is diagonalizable and its only possible eigenvalues are $0$ and $1$. The rank of $\overline{S}$ and 
$\ker \overline{S}$  only depend on $A$ and are independent of the choice of $m$ and the decomposition above. They will be denoted by $\er(A)$ and  $\uker(A)$. 
\end{lemma}

\begin{proof}
We start the proof by establishing an auxiliary statement for matrices over finite fields. 
Let $k$ be a finite field of characteristic $p$, and let $T\in \M_d(k)$. We claim that there exists $m\ge 1$ such that $T^m$
is diagonalizable over $k$ with eigenvalues only $0$ and $1$. Indeed, after extending scalars to an algebraic closure $\overline{k}$, we can put $T$ into Jordan normal form. Each Jordan block has the form $J=\lambda I+N$,  where $N$ is nilpotent and $N^d=0$. Since all eigenvalues of $T$ lie in some finite extension of $k$, every nonzero eigenvalue $\lambda$
has finite multiplicative order. Choose $m_0\ge 1$ such that $\lambda^{m_0}=1$ for every nonzero eigenvalue $\lambda$ of $T$.  
Choose $r$ large enough so that  $p^r\ge n$ and set $m= m_0 p^r$. In order to determine $T^m$, we compute $J^m$ for each Jordan block $J$ corresponding 
to an eigenvalue $ \lambda$. For $\lambda=0$, we have 
$J^m=N^m=0,$ because $m\ge d$. If $\lambda\neq 0$, then
\[
J^m=(\lambda I+N)^m
=\sum_{j=0}^{d-1}\binom{m}{j}\lambda^{m-j}N^j.
\]
Since $m$ is divisible by $p^r$ with $p^r\ge d$, we have
\[
\binom{m}{j}\equiv 0 \pmod p
\qquad
\text{for } 1\le j<d.
\]
Hence all the terms involving $N^j$ with $j\ge 1$ vanish, and so $J^m=\lambda^m I=I.$ Thus every Jordan block of $T^m$ is either $0$ or $I$. Therefore,
$T^m$ is  diagonalizable with eigenvalues only $0$ and $1$. By applying this claim to the reduction $\overline{A}\in M_d(\mathbb F_p)$, it follows that there exists $m\ge 1$ such that $\overline{B}:=\overline{A}^{\,m}$
is diagonalizable with eigenvalues only $0$ and $1$. In particular,
$\overline{B}^2=\overline{B}.$
Set $ B:=A^m.$ It follows that  $B^2-B\in pM_d(\mathbb Z_p).$ 
We now aim to construct an idempotent lift $S \in M_d(\mathbb Z_p)$ of $\overline{B}$ such that $B-S$ commutes with $S$. If one disregards this additional commutativity requirement, the existence of such an idempotent lift follows directly from Hensel's lemma. In the present setting, however, we will revisit the proof of Hensel's lemma and verify that the lifting procedure can be carried out in such a way that the desired commutativity condition is preserved. First note that $\overline{B}^2=\overline{B}$ implies 
\[
(2\overline{B}-I)^2
=
4\overline{B}^2-4\overline{B}+I
=
I.
\]
In particular, we know that  $2\overline{B}-I$ is invertible in $\M_d(\mathbb F_p)$. Lemma~\ref{lem:inv} implies that 
$2B-I$ is invertible over $\ZZ_p$. 

We now recursively define a sequence $(X_r)_{r \ge 0}$ of matrices in $\M_d(\ZZ_p)$ with the following properties:

\begin{enumerate}
\item For every $r\ge 0$, the matrix $X_r$ is a rational function in $B$.
\item For all $r \ge 0$, we have $ X_{r}^2-X_r \equiv 0 \pmod{ p^{2^r} }$.
\item For all $r \ge 0$, we have $X_{r+1} \equiv X_r \pmod{ p^{2^r}}$.
\item  $2X_r-I$ is invertible in $\M_d(\ZZ_p)$ for all $r \ge 0$. 
\end{enumerate}

Set $X_0=B$, and note that $X_0$ satisfies (1)--(4) above. Suppose for some value of $r \ge 0$, the matrix
$X_r \in \M_d(\ZZ_p)$ is constructed so that the above properties are satisfied. Let $ \overline{X_r}$ denote the reduction of $X_r$ modulo $p$. We note that 
\[
(2\overline{X_r}-I)^2
=
4\overline{X_r}^2-4\overline{X_r}+I
=
I.
\]
Lemma~\ref{lem:inv} implies that $2\overline{X_r}-I$ is invertible in $\M_d(\ZZ_p)$. 
It follows that $X_{r+1}$ defined by 
\begin{equation}\label{defx}
X_{r+1}= {X_r^2}{(2X_r-I)^{-1}}.  
\end{equation}
is an element of $\M_d(\ZZ_p)$. 
It is also clear that $X_{r+1}$ is given by a rational function in $B$.  We have thus established (1) for $r+1$.  Note also that 
  
\begin{equation}\label{x-1}
 X_{r+1}- I = (X_r^2- (2X_r- I ) )
(2X_r- I)^{-1}= (X_r- I)^2(2X_r-I)^{-1}.
\end{equation}

Multiplying \eqref{defx} and \eqref{x-1} yields
\[
X_{r+1}^2-X_{r+1}= (X_r^2-X_r)^2(2X_r-I)^{-2}.
\]

By the induction hypothesis we have that $X_r^2-X_r\in p^{2^r} \M_d(\mathbb Z_p)$, and that $2X_r-I$ is invertible in $\M_d(\mathbb Z_p)$. These imply that 
\[
X_{r+1}^2-X_{r+1}
\in
p^{2^{r+1}}\M_d(\mathbb Z_p),
\]
i.e., (2). A similar computation shows that,
\[
X_{r+1}-X_r
=
-(X_r^2-X_r)(2X_r-I)^{-1},
\]
which, in turn, implies that
\[
X_{r+1}-X_r
\in
p^{2^r}\M_d(\mathbb Z_p).
\]
Finally, claim (4) follows immediately from (3). 

Property (3) shows that $(X_r)_{r \in \mathbb{N}}$ is a Cauchy sequence in $\M_d(\mathbb Z_p)$. Since $\M_d(\ZZ_p)$ is a complete metric space, this sequence has a limit $S \in \M_d(\ZZ_p)$. Using property (2) and passing to the limit we deduce that $S^2=S$. In particular,  $\overline{S}$ is diagonalizable with eigenvalues only $0$ and $1$. Write $B=S+pM$ for some $M\in \M_d(\mathbb Z_p)$. It follows from (1) in the construction  that for every $r \ge 1$, we have $X_rB=BX_r$. Passing to the limit again we deduce that $SB=SB$, and hence $SM=MS$. 

Let us now show that the kernel and rank of $S$ are independent of the decomposition. Suppose 
$A^{m_1}=S_1+ pM_1$ and $A^{m_2}= S_2+ pM_2$ are decompositions as above. First, note that for every $k_1 \ge 1$, we have 
\[ A^{m_1k_1} = (S_1+ pM_1)^{k_1} = S_1^{k_1}+ pN_1= S_1+ pN_1,\]
for some matrix $N_1 \in \M_d(\ZZ_p)$. This implies that $A^{m_1k_1} \equiv S_1 \pmod{ p }$. 
By a similar reasoning, we have that for every 
$k_2 \ge 1$, we have $A^{m_2k_2} \equiv S_2  \pmod{ p }$. Now, set $k_1=m_2$ and $k_2=m_1$. We deduce that $A^{m_1m_2}$ is congruent to both $S_1$ and $S_2$ modulo $p$. This implies that 
$S_1 \equiv S_2 \pmod{ p }$. Especially, this shows that $S_1$ mod $p$ and $S_2$ mod $p$ have the same rank and the same kernel. This proves the last claim and finishes the proof of the theorem. 
\end{proof}

\begin{remark}\label{rem:con}
Note that, for $S \in M_d(\ZZ_p)$, the property of being a projection modulo $p$ depends only on the congruence class of $S$ modulo $p$. In other words, if $S \equiv T \pmod p$ and $S$ is a projection modulo $p$, then $T$ is also a projection modulo $p$. We will use this fact without further comment in the sequel.

\end{remark}

\subsection{Decomposition for idempotents}
The next lemma will be used in the proof of Proposition \ref{thm:intrinsic-kernel-idempotent-part}. Although standard over fields, we include a proof that remains valid in our setting.

\begin{lemma} \label{lem:submodules}
Let $S \in \M_d(\ZZ_p)$ satisfy $S^2=S$. Then, there exists a decomposition
$\ZZ_p^d= V^1 \oplus V^0$ into $\ZZ_p$-submodules such that $Sv=v$ for all $v \in V^1$ and $Sv=0$ for all $v \in V^0$.
\end{lemma}

\begin{proof}\label{lem:akb}
Let $V^1= \textrm{Im} (S)= \{ Sw: \, w \in \ZZ_p^d \}$ and $V^0:= \{ (I-S) w: \, w \in \ZZ_p^d \}$. It is easy to see that both $V^0$ and $V^1$ are $\ZZ_p$-submodules. Note that for every 
$v \in \ZZ_p^d$ we have the decomposition
\[ v= Sv+ (I- S) v \in V^1 + V^0.\]
We now show that $V^1 \cap V^0 = \{ 0 \}$. Suppose $v \in V^1 \cap V^0$. Then  $v= Sw_1= (I-S)w_2$. Hence, $$v= Sw_1= S^2w_1= S(I-S)w_2= 0.$$
Finally, for every $v \in V^1$ we have 
$Sv= S(Sw)= S^2w= Sw= v$ and for every $v \in V^0$, we have
$Sv= S( I- S)w= (S-S^2)w=0.$
This establishes the claim. 
\end{proof}

The next proposition characterizes $\ker S$ entirely in terms of $A$.

\begin{proposition}\label{thm:intrinsic-kernel-idempotent-part} 
Let $A\in \M_d(\ZZ_p)$, and suppose $m\geq 1$ is such that  $A^m$ is idempotent mod $p$ and has a decomposition $A^m=S+pM $ with $S,M\in M_d(\ZZ_p)$ satisfying $S^2=S$ and $SM=MS.$ Then 
 \[ \ker S=\{v\in \ZZ_p^d: A^n v\to 0 \text{ as } n\to\infty\}. \]
\end{proposition}

\begin{proof} 
Put $B=A^m$. Since $B=S+pM$ and $SM=MS$, both $\ker S$ and $\operatorname{Im} S$ are invariant under $B$.
Let $v\in \ker S$. Then \[Bv=(S+pM)v=pMv. \] Since $MS=SM$, the submodule $\ker S$ is also invariant under $M$. Hence, by induction, $ B^n v=p^n M^n v. $
Because $M^n v\in \ZZ_p^d$ for all $n$, we get $ B^n v\to 0. $ Now every sufficiently large power of $A$ has the form $ A^{qm+r}=A^rB^q$ with $ 0\leq r<m$. 
From here it follows that $ A^n v\to 0$ as $n \to \infty$.  This proves the inclusion 
\[ \ker S\subseteq \{v\in \ZZ_p^d: A^n v\to 0\}. \]

In order to prove the reverse inclusion, let $w\in \operatorname{Im} S$. Since $S$ is idempotent, we have $Sw=w$, which implies that $Bw=(S+pM)w=w+pMw.$ From here we deduce 
$Bw\equiv w \pmod p.$

By induction, $ B^n w\equiv w \pmod p $ for all $n\geq 1$. More generally, if $w\neq 0$, choose $k\geq 0$ maximal such that 
$ w\in p^k\ZZ_p^d. $ Then $w=p^k w_0$ with $w_0\notin p\ZZ_p^d$, and $ B^n w=p^kB^n w_0\equiv p^k w_0 \pmod {p^{k+1}}. $ 
Hence $B^n w$ has the same $p$-adic order as $w$ for every $n$. In particular, $ B^n w\nrightarrow 0.$ 
Now suppose $v \in \ZZ_p^d$ satisfies $A^nv \to 0$. Then in particular $B^n v=A^{mn}v\to 0.$
Write $v=v^1+ v^0$ as in Lemma~\ref{lem:submodules}.  
Both summands are $B$-invariant. By the above argument we have $v^1=0$. This implies that  $ v\in \ker S$, and completes the proof. 
\end{proof}

\section{Dynamics of  Multi-Variable Polynomial Maps} \label{sec:dynamics}
By a {\it polynomial map on} $\ZZ_p^d$, we refer to a function $g:\ZZ_p^d\to \ZZ_p^d$ that is given  coordinate-wise as $g=(g_1,\dots,g_d)$, where each $g_i\in \ZZ_p[X_1,\dots,X_d]$. Note that every  polynomial over $\ZZ_p$ can be expressed as a finite sum of monomials
\[
\sum_{\alpha} c_{\alpha}X^\alpha,
\]
where $\alpha=(\alpha_1,\dots,\alpha_d)\in \mathbb N^d$, $c_\alpha\in \ZZ_p$, and
$ X^\alpha=X_1^{\alpha_1}\cdots X_d^{\alpha_d}.$

\medskip

In this section, we are interested in studying the growth of the orbits of polynomial maps.  The study of the orbit of a polynomial map $g$ modulo $p^n$ reduces to studying the dynamics of $g$ at the $p$-adic scale $p^{-n}$. 
We start with the standard Taylor linear approximation. The proof is included for the sake of completeness.  

\begin{lemma}\label{taylor}
Let $g:\ZZ_p^d\to \ZZ_p^d$ be a polynomial map and $x \in \ZZ_p^d$. For every $n \ge 1$, and $t \in \ZZ_p^d$ we have
\[ g(x+p^n t)\equiv g(x)+p^n Dg(x)t\pmod {p^{2n}}. \]
\end{lemma}

\begin{proof}
Write $g$ coordinate-wise as $g=(g_1,\dots,g_d)$, where each $g_i\in \ZZ_p[X_1,\dots,X_d]$. Let $n\geq 1$.  
Fix $1\leq i\leq d$.  We will show that
\[ g_i(x+p^n t)\equiv g_i(x)+p^n \sum_{j=1}^d \frac{\partial g_i}{\partial X_j}(x)t_j \pmod {p^{2n}}. \]
To prove the coordinate-wise congruence, consider first a single monomial
\[ X^\alpha=X_1^{\alpha_1}\cdots X_d^{\alpha_d}. \]
We have
\[ (x+p^n t)^\alpha = \prod_{j=1}^d (x_j+p^n t_j)^{\alpha_j}.\]
By the binomial theorem,
\[
(x_j+p^n t_j)^{\alpha_j} = x_j^{\alpha_j} + \alpha_j x_j^{\alpha_j-1}p^n t_j +
\sum_{r=2}^{\alpha_j} \binom{\alpha_j}{r}x_j^{\alpha_j-r}p^{nr}t_j^r.
\]
Since $n\geq 1$, every term with $r\geq 2$ is divisible by $p^{2n}$. Therefore,
\[ (x_j+p^n t_j)^{\alpha_j} \equiv x_j^{\alpha_j} + p^n \alpha_j x_j^{\alpha_j-1}t_j \pmod {p^{2n}}.
\]
Multiplying these congruences for $j=1,\dots,d$, and keeping only the terms up to first order in $p^n$, we obtain
\[
(x+p^n t)^\alpha \equiv x^\alpha + p^n \sum_{j=1}^d \alpha_j x_1^{\alpha_1}\cdots x_j^{\alpha_j-1}\cdots x_d^{\alpha_d}t_j \pmod {p^{2n}}.
\]
The terms involving products of two or more first-order factors are divisible by $p^{2n}$. Thus
\[
(x+p^n t)^\alpha \equiv x^\alpha + p^n \sum_{j=1}^d \frac{\partial X^\alpha}{\partial X_j}(x)t_j \pmod {p^{2n}}.
\]
As each $g_i$ is a sum of  finitely many multi-indices $\alpha$ appearing in $g_i$, the claim follows. 
\end{proof}

The following special case allows us to linearize the local dynamics and understand the dynamics of $g$ in terms of the dynamics of an affine map. 

\begin{corollary}\label{linearization}
Let $g:\ZZ_p^d\to \ZZ_p^d$ be a polynomial map and $x \in \ZZ_p^d$ such that $g(x)\equiv x \pmod {p^n}$. Write
\[
A_n:=Dg(x), \qquad b_n:=\frac{g(x)-x}{p^n}
\]
Then one has
\begin{equation} \label{eq:Taylor_mod_pn+1}
g(x+p^n t)\equiv x+p^n(A_n t+b_n)\pmod {p^{n+1}},
\end{equation}
for $t\in \ZZ_p^d$. 
 Write $X_n$ for the set of all congruence classes modulo $p^{n+1}$ of points $y \in \ZZ_p^d$ with $y \equiv x \pmod{ p^n }$, or, equivalently, all balls of radius $p^{-(n+1)}$ included in the ball centered at $x$ of radius $p^{-n}$. Let $\tau: X_n \to (\ZZ/p\ZZ)^d $ be the bijection defined in 
Section~\ref{sec:prelim}. Then for every $x \in X_n$, we have $ g(x)= \tau^{-1} \circ \Phi \circ \tau(x)$, where
\[ 
\Phi: (\ZZ/p\ZZ)^d \to (\ZZ/p\ZZ)^d, \qquad  \Phi(t) := A_nt + b_n.
\]
\end{corollary}

\begin{proof}
We know from Lemma~\ref{taylor} that 
\[ g(x+p^n t)\equiv g(x)+p^n Dg(x)t\pmod {p^{2n}}. \]
The definition of $b_n$ implies that $g(x)=x+p^n b_n$. Hence, this becomes
\[ g(x+p^n t)\equiv x+p^n b_n+p^n Dg(x)t\pmod {p^{2n}}. \]
This proves the first claim.  The additional statement about conjugacy with the affine map follows immediately from the discussion in Section~\ref{sec:prelim}.
\end{proof}

Note that since $A_n \in \M_d(\ZZ_p)$, the congruence class modulo $p^{n+1}$ of the expression on the right-hand side of \eqref{eq:Taylor_mod_pn+1} depends {\it only} on the congruence class of $t \in \ZZ_p^d/p \ZZ_p^d$, which can be naturally identified with $(\ZZ/p\ZZ)^d$. Thus, the dynamics of $g$ modulo $p^{n+1}$  can be naturally identified with that of the affine map from $ (\ZZ/p\ZZ)^d$ to itself defined by $ \overline{t} \mapsto A_t \overline{t}+ b_n$, where $ \overline{t}$ is the class of $t \in \ZZ_p^d$ modulo $p \ZZ_p^d$. In the rest of this section, this identification will be frequently used without further remarks. 

\begin{lemma}\label{trivial}
Suppose $F: \ZZ_p^d \to \M_k(\ZZ_p)$ is a polynomial map with coefficients in $\ZZ_p$. Then $F$ is $1$-Lipschitz. In particular, if $x \equiv y  \pmod{ p^n }$ then $ F(x)  \equiv F(y) \pmod{ p^n }$. Furthermore, for every polynomial map $f: \ZZ_p^d \to \ZZ_p^d$, if 
 $x \equiv y  \pmod{ p^n }$, then $ Df(x)  \equiv Df(y) \pmod{ p^n }$.
\end{lemma}

\begin{proof}
The proof is identical to the statement for maps from $\ZZ_p^d$ to $\ZZ_p^d$. 
\end{proof}

\begin{corollary}\label{dlip}
Suppose $f: \ZZ_p^d \to \ZZ_p^d$ is a polynomial map and $x \in \ZZ_p^d$ is a fixed point for $f$ mod $p^\ell$ for some $\ell \ge 1$. If $y=f^m(x)$ for some $m \ge 1$, then  $Df(y) \equiv Df(x) \pmod{ p^\ell }$. 
\end{corollary}

\begin{proof}
Since $f(x) \equiv x  \pmod{ p^\ell }$, from a repeated application of Lemma~\ref{trivial} it follows that $y=f^m(x) \equiv x  \pmod{ p^\ell }$ for all $m \ge 1$. Applying the second part of Lemma~\ref{trivial} gives $Df(y) \equiv  Df(x) \pmod{ p^\ell }$.  
Hence by the chain rule,
\[ Df^n(y)= \prod_{i=0}^{n-1} Df( f^{i}(y)) \equiv   \prod_{i=0}^{n-1} Df( f^{i+m}(x)) \equiv \prod_{i=0}^{n-1} Df(x)= (Df(x))^n \pmod{p^\ell}. \]
\end{proof}

We now combine this simple observation with Lemma~\ref{lem:decom} to obtain the following.

\begin{lemma}\label{improvement}
Suppose $g$ is a polynomial map of $\ZZ_p^d$ and $x_0 \in \ZZ_p^d$ such that $g(x_0) \equiv x_0 \pmod{p} $. Then there exists $r \ge 1$ and $S, M \in \M_n(\ZZ_p)$ such that 
\begin{enumerate}
\item $ (Dg(x_0))^r= S+pM$, where $S^2=S$ and $SM=MS$. 
\item $Dg^r(x_0)$ is a projection modulo $p$, i.e. it satisfies $(Dg^r(x_0))^2 \equiv Dg^r(x_0) \pmod{p}$ . 
\end{enumerate}
\end{lemma}

\begin{proof}
Set $A=Dg(x_0) \in \M_d(\ZZ_p)$. By Lemma~\ref{lem:decom} there exists $r \ge 1$ such that $A^r$ can be written as
$A^r= S+pM$ with $SM=MS$ and $S^2=S$. In particular, $S$ is a projection mod $p$. For $j \ge 1$, set $ x_j:= g^j(x_0)$. Since $x_1  \equiv x_0 \pmod{ p }$ and $g$ is $1$-Lipschitz, it follows that $x_j \equiv x_0 \pmod{ p }$ for all $j \ge 1$. As the derivative map $Dg: \ZZ_p^d \to \M_d( \ZZ_p)$ is a polynomial map with coefficients in $\ZZ_p$,  it follows from Lemma~\ref{trivial} that $ Dg(x_j)  \equiv Dg(x_0) \pmod{ p }. $ By applying the chain rule we deduce 
\[ D g^r( x_0)\equiv \prod_{j=0}^{r-1} Dg(x_j) \equiv  \prod_{j=0} ^r Dg(x_0) \equiv  A^r \equiv S \pmod{p}. \]
In view of Remark~\ref{rem:con} the claim follows.
\end{proof}

\subsection{Strong fixed points and period growth}
To bound the size of the orbit of a polynomial map modulo $p^n$, we must control both $\Per_{f_n}(x)$ and $\Tail_{f_n}(x)$. In this section, we focus on the former. The following definition will play a crucial role in the proof.

\begin{definition}\label{strong}
Suppose $f: \ZZ_p^d \to \ZZ_p^d$ is a polynomial map and let $n \ge 2$.  We say that a point $x \in \ZZ_p^d$ is a strong fixed point of $f$ if 
\begin{enumerate}
\item $f(x)  \equiv x \pmod{ p^2 }$, or, equivalently, if $x$ is fixed by $f_{2}$. 
\item There exists $S, M \in \M_d(\ZZ_p)$ satisfying $S^2=S$ and 
$SM=MS$ such that $Df(x) \equiv S+pM \pmod{ p^2 }$.
\end{enumerate}
\end{definition}

The next lemma shows that the property of being a strong fixed point is stable under perturbations and passing to $f$-orbit points. 

\begin{lemma}\label{simplestrongfixedpoin}
Suppose $f: \ZZ_p^d \to \ZZ_p^d$ is a polynomial map and let $x \in \ZZ_p^d$ be a strong fixed point of $f$. Then
\begin{enumerate}
\item If $y \equiv x \pmod{p^2}$ then $y$ is a strong fixed point of $f$. 
\item For every $k \ge 1$, also $f^k(x)$ is a strong fixed point of $f$. 
\end{enumerate}
\end{lemma}

\begin{proof}
We first prove (1) in Definition \ref{strong}. Let $y \equiv x \pmod{p^2}$. Since $f$ is a polynomial map with coefficients in $\ZZ_p$, it preserves congruences modulo $p^2$. Hence $f(y) \equiv f(x) \pmod{p^2}$. Since $x$ is a strong fixed point, we have $f(x) \equiv x \pmod{p^2}$. Combining this with $x \equiv y \pmod{p^2}$, we obtain $f(y) \equiv y \pmod{p^2}$.

It remains to check condition (2) in Definition \ref{strong}. The entries of $Df$ are again polynomial functions with coefficients in $\ZZ_p$. Therefore, $y \equiv x \pmod{p^2}$ implies $Df(y) \equiv Df(x) \pmod{p^2}$. By assumption, there exist $S,M \in \M_d(\ZZ_p)$ such that $S^2=S$, $SM=MS$, and $Df(x) \equiv S+pM \pmod{p^2}$. Thus $Df(y) \equiv S+pM \pmod{p^2}$ as well. Hence $y$ is a strong fixed point of $f$.

We now move to the proof of part (2). Since $f(x) \equiv x \pmod{p^2}$ and $f$ is 1-Lipschitz, a simple induction shows that  $f^k(x) \equiv x \pmod{p^2}$ for all $k \ge 1$. Applying the first part with $y=f^k(x)$ shows that $f^k(x)$ is a strong fixed point of $f$
\end{proof}

Our first goal is to obtain a precise description of the growth of $\Per_{f_n}(x)$ when $x$ is a strong fixed point. This result will later be used to derive a weaker statement in the general case. Throughout this subsection, we use the notation introduced in Lemma~\ref{lem:decom}.

\begin{lemma}\label{lem:goup}
Suppose $h: \ZZ_p^d \to \ZZ_p^d$ is a polynomial map and let $x$ be a strong fixed point for $h$ satisfying $h(x) \equiv x \pmod{ p^n }$ for some $n \ge 2$. Write
$b_n(h,x)= \frac{h(x)-x}{p^n}$.  
Then the following dichotomy holds: 
\begin{enumerate}

\item If  
\begin{equation*}\label{CC}
\er(Dh(x))>0 \quad \textrm{and} \quad b_n(h, x) \not\in  \uker (Dh(x))
\end{equation*}
then $\Per_{h_{n+1}}(x)=p$ 
\item Otherwise, we have  $\Per_{h_{n+1}}(x)=1$. 
\end{enumerate}

\end{lemma}

\begin{proof} Let $X_n$ and the map $\tau$ be as in Corollary \ref{linearization}. Recall that $\tau (x)=0$. Hence, using Lemma~\ref{linearization}, it suffices to study the period of the point $0$ for the map $\Phi: (\ZZ/p\ZZ)^d \to (\ZZ/p\ZZ)^d$. 
Since $x$ is a strong fixed point of $h$, there exist matrices $S, M \in \M_d(\ZZ_p)$ where $S^2=S, SM=MS$ and $Dh(x) \equiv S+pM \pmod{ p^2 }$.
Let $\ZZ_p^d= V^0 \oplus V^1$ be the decomposition of $\ZZ_p^d$ into $\ZZ_p$-submodules provided by Lemma~\ref{lem:akb} and Lemma~\ref{lem:submodules}.
We write $b_n$ for $b_n(h,x)$.  Decompose $b_n=b_n^1+ b_n^0$ with $b^i_n \in V^i$, for $i=0,1$. For $t= t^1+ t^0$ with $t^i \in V^i$ we have
\begin{equation}\label{canonical}
 \Phi(t)= St + b_n = (t^1+b_n^1)+b_n^0.
\end{equation}
By Lemma~\ref{lem:decom} it holds that $\er(Dh(x))= \dim V^1$ and $b^1 =0$ if and only if $b \in  \uker (Dh(x))$. We now distinguish two cases:

In case (1), we know that $b^1_n \neq 0$ and $\dim V_1 >0$. It is easy to see that 
$\Phi^k(t)=t^1+kb_n^1+ b^0_n$. In particular, we have $\Phi^{p+1}(t)= \Phi(t)$. This proves the first part of the claim. 

Otherwise, either $\dim V^1=0$, in which case $\Phi(t)=b_n^0$, or $\dim V^1>0$ and $b_n^1=0$, in which case $\Phi(t)=t^1+b_n^0$. In either case, a direct computation shows that $\Phi^2=\Phi$, proving the second part of the claim.

\end{proof}

When calculating the period length of strong fixed points, we need to distinguish between odd primes and $p=2$ here for a technical reason which will become.

\begin{theorem} \label{lem:period_fixed_points}
Let $p$ be odd and suppose $f: \ZZ_p^d \to \ZZ_p^d$ is a polynomial map and $x \in \ZZ_p^d$ is a strong fixed point for $f$. Then either there exists $C$ such that $\Per_{f_n}(x) \le C$ for all $n \ge 1$, or there exists
$n_0$ such that for all $n \ge 1$ we have 
\[\Per_{f_{n}}(x)= p^{\max(n-n_0, 0)}.\]
\end{theorem}

\begin{proof}
Note that as the statement is about $\Per_{f_n}(x)$, by passing to a point in the orbit of $x$, we can always assume that $x$ is itself periodic mod $p^m$ for a given $m$. 
Since $x$ is a strong fixed point for $f$ we know that $f(x) \equiv x \pmod{ p^2 }$
and that $Df(x)=S+pM$ with $S^2=S$ and $SM=MS$. In the course of the proof we will use the fact that $Df$ along the $f$-orbit of $x$ is congruent to $S$ mod $p$. Note that $\Per_{f_2}(x)=1$. Let $m > 2$ be the least integer with the property that $ \Per_{f_m}(x)>1$. Hence $\Per_{f_{m-1}}(x)=1$. Since $x$ is a strong fixed point for $f$, it follows from Lemma~\ref{simplestrongfixedpoin} that $\Per_{f_{m}}(x)=p$. We will show that 

\medskip

\noindent
{\bf Claim.} For every $k \ge 1$, we have $\Per_{f_{m+k}}(x)=p^{k+1}$.

\medskip
We will prove the claim by induction on $k$. We will first prove it for $k=1$. Choose $y:=f^{p}(x)$  in the orbit of $x$ such that 
\[f(y) \equiv y \pmod{ p^{m-1} } \quad  \textrm{and} \quad  f^p(y) \equiv y \pmod{ p^m }.\]  

Note that $y$ is also a strong fixed point for $f$, and  $\Per_{f_{m}}(y)=p$. We also know by Lemma~\ref{lem:goup} that the following hold:
\[\er(Df(y))>0 \quad \textrm{and} \quad b_{m-1}(f,y) \not\in  \uker (Df(y)).
\]

Set $h=f^p$. Note that $h(y)  \equiv y \pmod{ p^m }$, that is, $\Per_{h_m}(y)=1$,  and our goal is to show that
$\Per_{h_{m+1}}(y)=p$.

In order to be able to apply Lemma~\ref{lem:goup}, we need to compute $Dh(y)$ and $b_m(h,y)$.  Since $S$ is an idempotent mod $p$, we have 
\[
    Dh(y)= D(f^p)(y) = \prod_{i=0}^{p-1} Df(f^i(y)) \equiv S^p \equiv S \pmod{ p }
\]

Set $v_j = f^j(y) - y$ for $j \geq 0$, so $v_0 =0$.  First note that since $ f(y) \equiv y \pmod{ p^{m-1} }$ and $f$ is 
$1$-Lipschitz, we deduce that $f^2(y) \equiv f(y) \pmod{ p^{m-1} }$. In particular, $f^2(y) \equiv  y \pmod{ p^ {m-1}}$. 
A simple induction shows that all $v_j$ are divisible by $p^{m-1}$. 
We claim 
\begin{equation}\label{eq:rec}
    v_j \equiv p^{m-1} \sum_{i=0}^{j-1} Dh(y)^i\, b_{m-1}(f,y) \pmod{p^{2m-2}}
    \quad \textrm{for all } 0 \leq j \leq p.
\end{equation}
We prove this equality by induction. For $j=1$, we have
\[v_1= f(y)-y = b_{m-1}(f, y) p^{m-1},\]
For the induction step, since $v_j $ is divisible by $p^{m-1}$, the Taylor expansion (Lemma~\ref{taylor}) of $f$ at $y$ gives 
\[
    v_{j+1} = f(y + v_j) - y
    \equiv (f(y) - y) + Df(y)\,v_j \pmod{p^{2m-2}},
\]
Using $f(y) - y = p^{m-1} b_{m-1}(f,y)$ and the induction hypothesis yields 
\[
    v_{j+1} \equiv p^{m-1} b_{m-1}(f,y) + Df(y)\cdot p^{m-1} \sum_{i=0}^{j-1}Df(y)^i b_{m-1}(f,y)
    = p^{m-1} \sum_{i=0}^{j}Df(y)^i \, b_{m-1}(f,y) \pmod{p^{2m-2}}.
\]
This completes the induction. Applying \eqref{eq:rec} with $j = p$ we obtain
 
\begin{equation}\label{trans}
p^m b_m(f^p, y)= f^p(y)- y \equiv v_p \equiv p^{m-1}  \! \left(\sum_{i=0}^{p-1}Df(y)^i\right) b_{m-1}(f,y)
    \pmod{p^{2m-2}}.
\end{equation}

Using $Df(y)= S+pM$, we have 

\begin{equation}
\begin{split}      
 \sum_{i=0}^{p-1}Df(y)^i \equiv    \sum_{i=0}^{p-1}(S+ pM)^i  & \equiv   \I+ (S+pM) +     \sum_{i=2}^{p-1} (S+pM)^i \\ 
    & \equiv \I+ S+ pM+ (p-2)S + p  \left( \frac{p(p-1)}{2}-1 \right) SM   \\
    & \equiv  p ( S+M-SM)+ \I - S  \pmod{p^2}. 
\end{split}
\end{equation}

We now claim that $  \sum_{i=0}^{j}Df(y)^i \not\equiv 0 \pmod{p^{2} }$ if $p \neq 2$ (the claim is not true for $p=2$ in general). Assuming the contrary, by first reducing modulo 
$p$ we deduce that $ \I- S \equiv 0 \pmod{ p }$. Write $\Id-S=pX$ for some
matrix $X \in \M_d(\ZZ_p)$, so that $S= \I- pX$. We claim that $X=0$. In fact, the equality $S^2=S$ implies that $pX=p^2X^2$ or $ X= pX^2$. Assuming that $X \neq 0$ we can write  $X=p^\ell Z$ with $\ell \ge 1$ and $Z \in \M_d(\ZZ_p)$ not divisible by $p$. Then we deduce $p^\ell  Z= p^{ 1+2 \ell} Z^2$
or $Z= p^{\ell+1} Z^2$, so that $Z$ is divisible by $p$, a contradiction. Hence $X=0$ and $S= \I$. But then the expression above simplifies to $p \I$, which is non-zero modulo $p^2$. The conclusion is that 
\[  
\sum_{i=0}^{j}Df(y)^i \not\equiv 0 \pmod{p^{2} }.
\]
Now, we need to show that $S b_m(f^p,y) 
\not \equiv 0 \pmod{ p }$. Assume the contrary. Using \eqref{trans}, $m>2$ and the facts that $Df(y) \equiv Df(x) \pmod{ p }$, and that $S$ commutes with $Df(x)$ mod $p$, we deduce 
\[ 0 \equiv p^m S b_m(f^p, y) \equiv
p^{m-1} \left(\sum_{i=0}^{p-1}Df(y)^i\right) Sb_{m-1}(f,y) \pmod{ p^{m+1} }
\]
Since the sum $(\sum_{i=0}^{p-1}Df(y)^i$ is not divisible by $p^2$, it follows that 
$Sb_{m-1}(f, y) \equiv 0 \pmod{ p }$, which contradicts our assumption. 

Summarizing the proof so far, we have shown that if $x$ is a strong fixed point for $f$, and
$m > 2$ is the least integer with the property that $ \Per_{f_m}(x)>1$, then $ \Per_{f_m}(x)=p$
and $\Per_{f_{m+1}(x)}=p^2$. Now suppose that the claim is proven for all values $k < K$ and we want to show that $ \Per_{f_{m+K}}(x)= p^{m+K}$. Consider the function $h=f^{p^{K-1} }$. Then 
$x$ is a strong fixed point for $h$ and for all $j <m$ we have  $\Per_{h_j}(x)= 
\Per_{f^{p^{K-1}}_j}(x)= 1$, i.e., we are in the situation of the beginning of the proof and can proceed inductively.

\end{proof}

For $p=2$ we will prove a slightly weaker result with an additional requirement which will nonetheless yield the desired bounds on the period length in the end.

\begin{lemma}\label{lem:p2-strong} Let $f:\mathbb \ZZ_2^d\to\mathbb \ZZ_2^d$ be a polynomial map and $z\in\mathbb \ZZ_2^d$. Suppose $f(z)\equiv z\pmod 4$ and $Df(z)\equiv S\pmod 4$ with $S^2 = S$. Then either the sequence $\Per_{f_n}(z)$ is bounded, or there are $n_0, n_1 \in \mathbb{N}$ with
\[
\Per_{f_n}(z)=2^{n-n_0}
\]
for all $n \geq n_1$.
\end{lemma}

\begin{proof} Put \(q_n := \Per_{F_n}(z)\). Since \(f(z) \equiv z \pmod{4}\), we have
\(q_1=q_2=1\). Moreover, it holds that
$q_{n+1}\in\{q_n,2q_n\}$, for all \(n\in\mathbb{N}\), and hence every \(q_n\) is some power of $2$. Now we assume that the sequence $(q_n)_{n\in\mathbb{N}}$ is not bounded.
Then we choose \(m\geq 2\) such that
$q_{m+1}=2q_m$, and we write \(q_m=2^r\). We may without loss of generality assume that
$z$ lies on the eventual cycle of $f$ modulo $2^{m+1}$. Then we set
$h:=f^{q_m}=f^{2^r}.$ By definition, we have
$ h(z)\equiv z\pmod{2^m},\qquad h(z)\not\equiv z\pmod{2^{m+1}},\qquad h^2(z)\equiv z\pmod{2^{m+1}}. $
Thus, we may write
\[
h(z)-z=2^m b,\qquad b\not\equiv 0\pmod{2}.
\]
Since $f(z)\equiv z\pmod{4}$, every point in the $f$-orbit of $z$
is congruent to $z$ modulo $4$. Using Lemma \ref{trivial}, we have
\[
Dh(z)
 =\prod_{i=0}^{q_m-1}Df\bigl(f^i(z)\bigr)
 \equiv S^{q_m}
 \equiv S
 \pmod{4},
\]
because $S^2=S$. Now we apply Taylor expansion (Lemma~\ref{taylor}) to obtain
\begin{align*}
h^2(z)-z
 &=h(z+2^m b)-z\\
 &\equiv h(z)-z+Dh(z)\,2^m b
   \pmod{2^{m+2}}\\
 &=2^m\bigl(\operatorname{Id}+Dh(z)\bigr)b
   \pmod{2^{m+2}}.
\end{align*}
Since \(h^2(z)\equiv z\pmod{2^{m+1}}\), it follows that
$\bigl(\operatorname{Id}+Dh(z)\bigr)b\equiv 0\pmod{2}.$
Using $Dh(z)\equiv S\pmod{2}$, we get
$(\operatorname{Id}+S)b\equiv 0\pmod{2}.$ Since the residue field has characteristic \(2\), this is equivalent to $Sb\equiv b\pmod{2}.$ In particular,
$Sb\not\equiv 0\pmod{2}.$ We claim that from this level onward the period doubles at every level,
and prove this by induction. Suppose that, for some \(k\geq 0\), with $h_k:=h^{2^k},$ we have
\[
h_k(z)-z=2^{m+k}b_k,\qquad
Sb_k\not\equiv 0\pmod{2},
\]
and
\[
Dh_k(z)\equiv S\pmod{4}.
\]
For $k=0$, this is true by what we have just shown. Since
$h_k(z)\equiv z\pmod{2^{m+k}}$ and $m+k\geq 2$, we have
$ h_k(z)\equiv z\pmod{4}.$ Therefore,
$Dh_k\bigl(h_k(z)\bigr)\equiv Dh_k(z)\pmod{4}.$ Hence,
\[
Dh_{k+1}(z)
 =D(h_k^2)(z)
 =Dh_k\bigl(h_k(z)\bigr)Dh_k(z)
 \equiv S^2
 =S
 \pmod{4}.
\]
Another application of Lemma~\ref{taylor} yields
\[
h_{k+1}(z)-z
 =h_k^2(z)-z
 \equiv
 2^{m+k}\bigl(\operatorname{Id}+Dh_k(z)\bigr)b_k
 \pmod{2^{m+k+2}}.
\]
Using \(Dh_k(z)\equiv S\pmod{4}\), we obtain 
\begin{align*}
S\bigl(h_{k+1}(z)-z\bigr)
 &\equiv
 2^{m+k}S(\operatorname{Id}+S)b_k
 \pmod{2^{m+k+2}}\\
 &=2^{m+k}(S+S^2)b_k
 \pmod{2^{m+k+2}}\\
 &=2^{m+k+1}Sb_k
 \pmod{2^{m+k+2}}.
\end{align*}
Since $Sb_k\not\equiv 0\pmod{2}$, this implies that
$h_{k+1}(z)-z$ is not divisible by $2^{m+k+2}$, but, for
$w:=h_{k+1}(z)-z, $
the vector \(Sw\) has exact valuation \(m+k+1\). Thus, the lift from
level \(m+k+1\) to level \(m+k+2\) is non-trivial. Consequently, the
period doubles at the next level. This proves inductively that
$q_{m+k}=q_m2^k$
for all \(k\geq 0\). Hence, the claim follows.
\end{proof}

We can now prove the following theorem, which is part (1) of Theorem~\ref{them:finitary}.

\begin{theorem}\label{thm:period}
Let $f: \mathbb{Z}_p^d \to \mathbb{Z}_p^d$ be a polynomial map, and let
$x \in \mathbb{Z}_p^d$ be arbitrary. Then either there exists a constant $C$ such that $ \Per_{f_n}(x) \le C$ for all $n \ge 1$, or there exist  $C_1,C_2>0$ such that for all $n \ge 1$ we have 
\[ C_1 p^n \le \Per_{f_n}(x) \le C_2 p^n. \]

\end{theorem}

\begin{proof}
We first claim that there exists $m \ge 1$ and an element $y$ in the $f$-orbit of $x$, such that $y$ is a strong fixed point for $g:=f^{m}$. 
Set $n_1=\Tail_{f_2}(x)$  so that $f_2^{n_1}(x)$ is a periodic point for $f_2$. Hence, there exists $m_0 \ge 1$ such that for 
$h:=f^{m_0}$ and $y:=f^{n_1}(x)$ we have $h(y) \equiv y \pmod{ p^2 }$.  Let $A:=Dh(y)$. It follows from Corollary~\ref{dlip}
that $Dh( h^j(y) ) \equiv A \pmod{ p^2 }$. 
By Lemma~\ref{lem:decom},  there exist $ m_1 \ge 1$ such that $A^{m_1}$ can be expressed as 
$A^{m_1}= S+pM$ with $S^2=S$ and $MS=SM$. Now,
\[Dh^{m_1}(y)= \prod_{j=0}^{m_1-1} Dh(h^{j}(y) \equiv A^{m_1}
\equiv S+pM, \pmod{ p^2 }.\]
Setting $m:=m_0 m_1$, it follows that $y$ is a strong fixed point for $g:=h^{m_1}=f^{m}$.
We first consider the case $p>2$. Then applying Theorem~\ref{lem:period_fixed_points} to $g$ and the point $x$, we deduce that either  $\Per_{g_n}(x)=1$ for all $n \ge 1$ or there exists $n_0$ such that  $\Per_{g_n}(y)=p^{n-n_0}$ for all $n \ge n_0$. Applying Lemma~\ref{sizeofperiod} to $S=\ZZ/p^n\ZZ$ and the mapping $f_n$ we deduce 
\[ \Per_{f_n}(x) \le n_1 \Per_{g_n}(y) \le n_1.  \]
In the second case, we have 
\[ \Per_{f_n}(x) \le n_1 \Per_{g_n}(y) =n_1 p^{n-n_0}, \quad n \ge n_0.\]
and
\[  \Per_{f_n}(x) \ge \Per_{f^{n_1}_n}(y)= p^{n-n_0}, \quad n \ge n_0.\]
The claim follows immediately. If $p=2$, we consider $k:=h^2$ and we obtain $k(y) \equiv y \pmod 4$. Moreover, 
\[
Dk(y)=Dh(h(y))Dh(y) \equiv Dh(y)^2 \equiv (S+2M)^2 \pmod 4.
\]
Since $S^2 = S$ and $SM = MS$ we obtain
\[
(S+2M)^2 = S^2+2SM+2MS+4M^2 = S+4SM+4M^2 \equiv S \pmod 4.
\]
Thus $Dk(y) \equiv S \pmod 4$ and we are in the situation of Lemma~\ref{lem:p2-strong}. It remains only to compare the periods for $k=f^{2m}$ and for $f$. Put
\[
P_n:=\Per_{k_n}(y), \qquad Q_n:=\Per_{f_n}(y).
\]
Since $k=f^{2m}$, we have $P_n=\frac{Q_n}{\gcd(Q_n,2m)}$ by Lemma~\ref{lem:decom}. Therefore, $Q_n$ is bounded if and only if $P_n$ is bounded, and in the unbounded case $Q_n$ satisfies the claimed bound. Finally, since $y$ is a forward iterate of $x$, the eventual periods of $x$ and $y$ modulo $2^n$ are the same. Hence the same dichotomy holds for $\Per_{f_n}(x)$, completing the proof when $p=2$.
\end{proof}

\subsection{Growth of the pre-periodic part}
Recall that for a polynomial map $f: \mathbb{Z}_p^d \to \mathbb{Z}_p^d, x \in  \ZZ_p^d$ and $n \geq 1$, we write $\Tail_{f_n}(x)$ for the least integer $g(n)$ such that 
\[
f_n^k(x_0) = f_n^{m+k}(x_0)
\]
holds for $k \geq g(n)$ if $m=\Per_{f_n}(x)$. It is clear that $g(n+1) \geq g(n)$. The next theorem shows that the sequence $g(n)$ grows at most linearly. 

\medskip

\begin{theorem}\label{thm:preperiod-growth} Let $f:\ZZ_p^d\to\ZZ_p^d$ be a polynomial map with coefficients in $\ZZ_p$, and let $x\in\ZZ_p^d$. For $j\geq 0$, write $x_j:=f^j(x)$. For each $n\geq 1$, set $g(n):= \Tail_{f_n}(x)$. Write $r:=\Per_{f_1} (x)$. Then, for every $n\geq 1$ we have
\[ g(n+1)\leq g(n)+rd \leq g(n) + dp^d.\]
\end{theorem}

\begin{proof}
Set $m:=g(n)$, and let $T =\Per_{f_n} (x_m)$. Thus
$x_{j+T}\equiv x_j\pmod{p^n}$ for every $j\geq m$.
Note that $\widetilde{u}_j\in\ZZ_p^d$ defined by 
$x_{j+T}=x_j+p^n\widetilde{u}_j$ belongs to $\ZZ_p^d$. 

For $j\geq m$, define
\[
u_j:= \widetilde{u}_j \pmod p\in\mathbb F_p^d, \quad
\text{and} \quad A_j:=Df(x_j)\pmod p. \]
By Corollary~\ref{linearization} we can write 
\[
f(x_j+p^n\widetilde{u}_j)
\equiv
f(x_j)+p^nDf(x_j)\widetilde{u}_j
\pmod{p^{n+1}}.
\]
Applying the definition $x_{j+1}=f(x_j)$ and $x_{j+T+1}=f(x_{j+T})$, we obtain
\[
x_{j+T+1}-x_{j+1}
\equiv
p^nDf(x_j)\widetilde{u}_j
\pmod{p^{n+1}}.
\]
This implies 
\[ \frac{x_{j+T+1}-x_{j+1}}{p^n} \equiv Df(x_j)\widetilde{u}_j \pmod p.
\]
Hence by the definition $u_j$, we obtain the recursive formula 
\[
u_{j+1}=A_j u_j.
\]
Moreover, since $x_{j+r}\equiv x_j\pmod p$ for every $j\geq m$, we deduce $A_{j+r}=A_j$ holds for all $j \ge m$.   Set $ B:=A_{m+r-1}\cdots A_m.$ Then a repeated application of the recursion gives
\[
u_{m+\ell r}=B^\ell u_m
\]
for every $\ell\geq 0$. Consider the descending sequence of subspaces defined by 
\[
\mathbb F_p^d\supseteq \operatorname{Im}(B)\supseteq
\operatorname{Im}(B^2)\supseteq\cdots.
\]
We first claim that if for some $i \ge 1$ we have $\operatorname{Im}(B^i)= \operatorname{Im}(B^{i+1})$, then  $\operatorname{Im} (B^i)= \operatorname{Im}(B^j)$ holds for all $j \ge i$. To see this notes that if $v= B^{j}w$ then $v= B^{j-i}B^{i}w$ for some $w$. Since $\operatorname{Im}(B^i)= \operatorname{Im}(B^{i+1})$, there exists $w'$ with $B^iw = B^{i+1} w'$. This implies that $ x= B^{j+1} w'$, proving the claim. Hence, 
$\operatorname{Im}(B^d)= \operatorname{Im}(B^{d+1}).$ If this is not the case, the claim we just proved shows that $\dim \operatorname{Im}(B^i) \ge \dim \operatorname{Im}(B^{i+1})+1$ for all 
$0 \le i \le d$, which is impossible.

It follows that  for 
$W:=\operatorname{Im}(B^d)$ one has $B(W)=W$. In particular, the restriction $B|_W$ is an automorphism. Set $k:=m+rd$ and $z:=x_k.$ Then $ u_k=B^d u_m\in W.$ Since $k\equiv m\pmod r$, the points $z$ and $x_m$ occupy the same position on the cycle modulo $p$. Now let $F:=f^T$. Since $T=qr$, the chain rule and the periodicity of the matrices $A_j$ give $ DF(z)\equiv B^q\pmod p.$ Furthermore, $ F(z)=x_{k+T}\equiv z+p^n u_k\pmod{p^{n+1}}.$ Identify the fiber above $z\bmod p^n$ with $\mathbb F_p^d$ by sending $t\in\mathbb F_p^d$ to $z+p^n t\bmod p^{n+1}$. Under this identification, the map induced by $F$ is $J(t):=B^q t+u_k.$ Indeed, using Lemma~\ref{taylor} we have 
\[
F(z+p^n t)\equiv z+p^n\bigl(B^q t+u_k\bigr)
\pmod{p^{n+1}}.
\]
Since $u_k\in W$ and $B^q(W)=W$, the affine map $J$ preserves $W$. Its linear part is invertible on $W$, so $J|_W$ is a permutation of the finite set $W$. In particular, $0\in W$ is periodic under $J$. Hence there exists $s\geq 1$ such that $J^s(0)=0$. Translating this back to the orbit of $z$ gives $F^s(z)\equiv z\pmod{p^{n+1}}$, or equivalently, $f^{sT}(x_k)\equiv x_k\pmod{p^{n+1}}.$ Thus, $x_k$ is periodic modulo $p^{n+1}$. This proves the first claim.  Finally, $r\leq p^d$, since the cycle modulo $p$ is contained in the finite set $\mathbb F_p^d$ and $k=m+rd$. Hence, $g(n+1)\leq g(n)+dp^d$.
\end{proof}

\subsection{Proof of Theorem~\ref{thm:Hausdorff}}

\begin{proof}[Proof of Theorem~\ref{thm:Hausdorff}]
We will use Theorem~\ref{them:finitary} to prove suitable bounds for the size of 
$\pi_n(E)$. Denote the orbit of $x$ by
$ O:=\{ f^{(j)}(x_0):j\geq 0\}$.
Because the target $(\ZZ/p^n\ZZ)^d$ is finite and discrete, taking the closure of a subset of $\ZZ_p^d$ does not change its image under $\pi_n$. Therefore,
$ \pi_n(E) = \pi_n(O).$
This orbit consists of its pre-periodic part followed by its periodic part. Consequently,
\[
|\pi_n(O)| =
\Tail_{f_n}(x_0)+\Per_{f_n}(x_0).
\]
Note that since $\Tail_{f_n}(x_0)$ is growing at most linearly, we only need to consider the growth of periods. 
If $\Per_{f_n}(x_0)$ is bounded, then clearly, we have $|\pi_n(O)|= O(n)$ and 
hence $\dim_B(E)=0$. Otherwise, it follows from Theorem~\ref{them:finitary} that the sequence $ |\log \Per_{f_n}(x_0)- n \log p| $ is bounded. The claim follows immediately. 
\end{proof}

\section{Construction of Examples and $S$-adic Variations} \label{sec:examples}

Since there do not exist any ergodic polynomial maps on $\mathbb{Z}_p^d$ for $d>1$, the final aim of this article is to describe a way to obtain examples of ergodic maps in this setting. In view of Theorem~\ref{thm:ergodic:1d}, this yields examples of low-discrepancy sequences in $\ZZ_p^d$. Indeed, the following result implies that ergodic maps on $\mathbb{Z}_p^d$ can actually all be found by analyzing the one-dimensional case. 
\begin{theorem} \label{thm:ergodic_maps_Td} Let $T_d: \mathbb{Z}_p^d \to \mathbb{Z}_p^d$ be a continuous map. Then $T_d$ is ergodic if and only if it is conjugate to an ergodic function $T: \mathbb{Z}_p \to \mathbb{Z}_p$.
\end{theorem}
\begin{proof}[Proof of Theorem~\ref{thm:ergodic_maps_Td}]
A homeomorphism $\phi: \mathbb{Z}_p \to \ZZ_p^d$ is given by $z = \sum_{j=0}^\infty a_jp^j \mapsto \left(x_1,x_2,\ldots,x_d \right)$ with
\[
x_k = \sum_{j=0}^\infty a_{k-1+j\cdot d} p^{j}.
\]
If $T_d$ is ergodic, then the map $\phi \circ T_d \circ \phi^{-1}$ is an ergodic map on $\mathbb{Z}_p$. On the other hand if $T$ is ergodic, then $\phi^{-1} \circ T_d \circ \phi$ is an ergodic map on $Z_p^d$ by the definition of the Haar measure.
\end{proof}

Together with Theorem~\ref{thm:ergodic:1d}, this rather simple observation allows us to construct explicit examples of low-discrepancy in the one- and multi-dimensional setting.

One main motivation to study $p$-adic discrepancy lies in the fact that there are important connections between the discrepancy of real and $p$-adic sequences. Recall that the Monna map 
$\textrm{Mon}_p: \ZZ_p \to [0,1]$ is defined via 
\[ \textrm{Mon}_p  \left( \sum_{i=0}^{\infty} a_i p_i \right) = \sum_{i=0}^{\infty} a_i p_{-i-1}, \qquad 
a_i \in \{ 0, 1, \dots, p-1 \}. \]

It was proven by Meijer \cite{Mei68} that $(x_n)_{n \in \NN}$ is a $p$-adic low-discrepancy sequence if and only if the sequence $(\textrm{Mon}_p(x_n))_{n \in \NN}$ of real number has low-discrepancy. However, the induced maps on $[0,1]$ coming from algebraically defined maps on $\ZZ_p$ do not have any specific features that make them easier to study. Hence, it is preferable to study these sequences in $\ZZ_p$ directly, as is also the approach of our article. 

Another reason for studying ergodic $1$-Lipschitz transformations on $\ZZ_p^d$ is that they are promising candidates for the construction of new pseudorandom number generators. Indeed, ergodicity implies that every state is visited uniformly over time. The induced transformation modulo $p^n$ consists of a single cycle, guaranteeing the maximal possible period length $p^n$. Moreover, the algebraic structure of $\ZZ_p$ enables the construction of highly nonlinear yet computationally efficient state-transition functions, while still allowing rigorous mathematical analysis of their dynamical and distributional properties. Research on the application of ergodic $p$-adic transformations to pseudorandom number generation was pioneered in \cite{Ana04, Ana10}, see also \cite[Chapter~14]{Ana25} for a recent comprehensive overview. This potential application furthermore provides motivation for now presenting explicit examples of ergodic $1$-Lipschitz maps.

\subsection{One-Dimensional Setting} The obvious question to ask is how to identify ergodic $1$-Lipschitz functions. To do this we rely on \cite{Mah58}, where Mahler showed that every continuous function $f : \mathbb{Z}_p \to \mathbb{Q}_p$ can be written as a series
\[
f(x) = \sum_{n=0}^\infty a_n \binom{x}{n}
\]
with $a_n \to 0$, where the binomial coefficient is defined as
\[
\binom{x}{n} = \frac{x(x-1)\cdot\ldots\cdot (x-n+1)}{n!}.
\]
It is possible to give a characterization of $1$-Lipschitz functions on the $p$-adic integers as well as sufficient conditions for ergodicity in terms of the Mahler expansion. These conditions originally go back to Anashin in \cite{Ana94}, see also \cite{Jeo13} for a more recent summary of the results.
\begin{theorem}[Anashin] \label{thm:Anashin_Coefficients} Let $f: \mathbb{Z}_p \to \mathbb{Z}_p$ be a continuous function with Mahler expansion
\[
f(x) = \sum_{n=0}^\infty a_n \binom{x}{n}.
\]
Then $f$ is $1$-Lipschitz if and only if $|a_n| \leq |p|^{\lfloor\log_p(n)\rfloor}$ for all $n \geq 0$. Moreover, if the following conditions are satisfied, then $f$ is ergodic
\begin{align*}
& a_0 \not \equiv 0 \pmod{p},\\
& a_1 \equiv 1 \pmod {p}, \quad p \neq 2,\\
& a_1 \equiv 1 \pmod {4}, \quad p = 2,\\
& a_n \equiv 0 \pmod {p^{\lfloor\log_p(n+1)\rfloor+1}}, \quad n\geq2.
\end{align*}
\end{theorem}
Using these, we can easily construct some one-dimensional examples.
\begin{example} \label{ex:polynomials} At first, we re-consider the polynomial case: Besides the linear examples $T_1(x)= x +b$ with $b \not \equiv 0 \pmod {p}$ mentioned already in Section~\ref{sec:ud_and_discrepancy}, we can e.g. derive that the quadratic polynomial $T_2(x) = p^{\lfloor \log_p(2)\rfloor+1} \cdot \frac{x(x-1)}{2} + x + b$ is an ergodic map (if $p>3$). Hence each orbit of $T_2(x)$ is a low-discrepancy sequence by Theorem~\ref{thm:ergodic:1d}. At the same time, it must hold that $T_2(x)$ is a permutation polynomial: since $T_2^i(0)$ for $i=1,\ldots,p^2$ attains all residue classes $\bmod\ p^2$ and $T_2(x)$ is compatible, also $T_2(i)$ for $i=1,\ldots,p^2$ must attain all residue classes. Hence, $T_2(i)$ must be a permutation polynomial $\bmod\ p^2$. 
\end{example}
Note that the polynomial $T_2(x)$ from Example~\ref{ex:polynomials} does not appear in \cite[Theorem~2.8]{Wei25}, since it is not normalized, that is, its leading coefficient is not equal to~$1$. Moreover, observe that $\lfloor \log_p(n+1) \rfloor + 1 \geq 2$ for all $n \geq p-1$. Consequently, all examples arising from Theorem~\ref{thm:Anashin_Coefficients} are, modulo~$p^2$, essentially of the form $T_2(x)$. Finally, by Theorem~\ref{thm:permutation}, the reduction modulo~$p^2$ is decisive for obtaining low-discrepancy orbits: the permutation property modulo~$p^2$ already implies the permutation property modulo~$p^n$ for arbitrary~$n$.
\medskip

When starting the work on this article, the initial aim was to construct a $p$-adic low-discrepancy sequence, which does not build upon polynomials. Also this is achievable by applying Theorem~\ref{thm:Anashin_Coefficients}. 
\begin{example} To give a specific example, we can consider the orbit of the function
\[
f(x) = 1 + x + \sum_{n=2}^\infty p^{\lfloor\log_p(n+1)\rfloor+1} \binom{x}{n}.
\]
For this non-polynomial examples, the orbit values cannot be computed exactly because the defining sums contain infinitely many terms. Their reductions modulo $p^k$, however, can be computed easily, since all but finitely many coefficients satisfy $a_n\equiv 0\pmod{p^k}$.

\end{example}

\subsection{Multi-Dimensional Setting $\mathbb{Z}_p^d$} The homeomorphism $\phi: \mathbb{Z}_p \to \ZZ_p^d$ is given by 
\[
z = \sum_{j=0}^\infty a_jp^j \mapsto \left(x_1,x_2,\ldots,x_d \right),
\]
where
\[
x_k = \sum_{j=0}^\infty a_{k-1+j\cdot d} p^{j}.
\]
Hence, the map $g=\phi \circ T_1 \circ \phi^{-1}$ with $T_1(x) = x+1$ is an ergodic map on $\ZZ_p^d$. In the special case $d=2$, it is easy to describe what the map $g$ actually does. If we have an element $z=(a,b) \in \mathbb{Z}_p^2$ with $a = \sum_{i=0}^\infty a_i p^i, b = \sum_{i=0}^\infty b_i p^i$, then $\phi^{-1}(a,b) = \sum_{i=0}^\infty a_i p^{2i} + \sum_{i=0}^\infty b_i p^{2i+1} =: \sum_{i=0}^\infty c_i p^i$. Next, $T_1(x)$ increases the coefficient $c_0$ by $1$ if $c_0 < p-1$. If $c_0 = p-1$, then $c_0$ becomes $0$ and the same logic applies to $c_1$ and inductively to all other coefficients. Afterwards, the coefficients get split up into the odd and even ones by $\phi$ to obtain an element in $\mathbb{Z}_p^d$. If we look at the orbit of $(0,0)$ under this map, the first component of the $n$-th element in the orbit is in other words given by $\sum_{i=0}^\infty n_{2i} p^{i}$ if the $p$-adic representation of $n-1$ equals $n-1=\sum n_ip^i$. Similarly the second component is then $\sum_{i=0}^\infty n_{2i+1} p^i$. \\[12pt]
According to \cite[Section~6]{Ana06}, for any $1$-Lipschitz function $g: \mathbb{Z}_p \to \mathbb{Z}_p$, the function $f(x) = 1 +x +p(g(x+1)-g(x))$ is always ergodic on $\mathbb{Z}_p$. This observation and Theorem~\ref{thm:ergodic_maps_Td} allow for the construction of more examples in the same manner as we just described for $T_1(x)=x+1$. Instead of using $T_1(x) = x+1$, the ergodic map $T: \mathbb{Z}_p \to \mathbb{Z}_p$ may thus also be chosen as a polynomial of degree $d$. The latter are ergodic if and only if the mapping $z \mapsto f(z) \pmod {p^{\log_p(d)+3}}$ is compatible and transitive on the residue class $\mathbb{Z}_p/p^{\log_p(d)+3}\mathbb{Z}_p$, compare again \cite[Section~6]{Ana06}.

\subsection{$S$-adic variations} The simplest way to find multi-dimensional low-discrepancy sequences is to work with different prime bases $p_1,\ldots,p_d$ simultaneously instead of looking at $\ZZ_p^d$ as has already been realized in \cite{Mei68}. To conclude this paper, we shortly explain the approach. Let $S = (p_1,\ldots,p_d)$ be a vector of distinct prime numbers and consider the ring
\[
\mathbb{Z}_S := \mathbb{Z}_{p_1} \times \mathbb{Z}_{p_{2}} \times \ldots \times \mathbb{Z}_{p_d}.
\]
For a vector $K=(k_1,\ldots,k_d)$ of non-negative integers and $Z = (z_1,\ldots,z_d) \in \mathbb{Z}_S$ consider the neighborhoods
\[
\textrm{Disc}_S(Z,K) :=  \textrm{Disc}_{p_1}(z_1,k_1) \times \ldots \times \textrm{Disc}_{p_d}(z_d,k_d).
\]
in $\mathbb{Z}_S$. The normalized Haar measure on $\mathbb{Z}_S$ is denoted by $\mu$ so that $\mu(\textrm{Disc}_S(Z,K)) = \prod_{i=1}^d p_i^{-k_i}$. For a sequence $(x_n)$ in $\mathbb{Z}_S$ the expression
\[
\delta_N^{(P)}(x_n) := \sup_{Z \in \mathbb{Z}_S, k \in K} \left| \frac{\#\left( \textrm{Disc}_S(Z,K) \cap \{ x_1,\ldots,x_N \}\right)}{N} - \mu\left(\textrm{Disc}_S(Z,K) \right) \right|
\]
represents the $P$-adic discrepancy. The low-discrepancy property from Theorem~\ref{thm:ergodic:1d} can then easily be transferred to $Z_S$.
\begin{proposition} \label{prop:multidim} Let $p_1,\ldots,p_d$ be distinct prime numbers and let $f_i: \mathbb{Z}_p \to \mathbb{Z}_p$ be ergodic $1$-Lipschitz functions for $i=1,\ldots,d$. For $x_0 = (x_0^1,\ldots,x_0^d) \in \mathbb{Z}_S$  the sequence
\[
(x_n) := \left(f_1^{(n)}(x_0^1),\ldots,f_d^{(n)}(x_0^d)\right)
\]
is a low-discrepancy in $\mathbb{Z}_S$.
\end{proposition}
\begin{proof} Let $N \in \mathbb{N}$ be arbitrary and consider an arbitrary $P$-adic disc $\mathrm{Disc}_S(Z,K)$. For each component of the disc we know that $\mathrm{Disc}_{p_i}(z_i,k_i)$ contains $r_i= \lfloor N/p_i^{k_i} \rfloor$ or $r+1$ elements. By the Chinese remainder theorem, $\mathrm{Disc}_S(Z,K)$ thus contains $\lfloor N/ \prod_{i=1}^d p_i^{k_i} \rfloor$ or $\lfloor N/ \prod_{i=1}^d p_i^{k_i} \rfloor + 1$ elements. Therefore, $(x_n)_{n \in \mathbb{N}}$ is a $P$-adic low-discrepancy sequence.
\end{proof}
In contrast to $\mathbb{Z}_p^d$, it is now immediate from the one-dimensional case that there are ergodic affine linear maps on $Z_S$.
\begin{corollary} Let $p_1,\ldots,p_d$ be distinct odd prime numbers with, $(a_1,\ldots,a_d) \in Z_S$ with $a_i \equiv 1 \pmod{ p_i}$, $b=(b_1,\ldots,b_d)\in Z_S$ with $b_i$ relatively prime to $p_i$. Set
\[
A = \begin{pmatrix} a_1& 0 & 0 & \ldots &0\\
0 & a_2 & 0 & \ldots &0 \\
0 & 0 &\ddots &  & 0 \\
\vdots & \vdots & & \ddots & \vdots\\
0 & 0 & \ldots & 0 &a_d\end{pmatrix}.
\]
Then $T(x)=Ax+b$ is ergodic on $\mathbb{Z}_S$.
\end{corollary}

\bibliographystyle{abbrv}

\bibliography{literatur}

\end{document}